\documentclass[11pt]{amsart}
\usepackage{geometry}
\usepackage{wrapfig}
\usepackage{amsmath}
\usepackage{amssymb}
\usepackage{algorithm}
\usepackage{algpseudocode}
\usepackage{hyperref}
\hypersetup{colorlinks}
\usepackage{caption}
\usepackage{enumitem}

\usepackage{graphicx,verbatim,lineno,titletoc}
\usepackage{adjustbox}
\usepackage{calc}
\usepackage{ytableau}
\usepackage{array}
\usepackage{tikz}
\usetikzlibrary{shapes.geometric, arrows}
\usepackage{multirow}

\usepackage{pgfplots}
\usepackage{wrapfig}
\usepackage{tabto}
\usepackage{stmaryrd}
\usepackage[capitalize,nameinlink,noabbrev]{cleveref}
\theoremstyle{plain}
\newtheorem{theorem}{Theorem}[section]
\newtheorem{lemma}[theorem]{Lemma}
\newtheorem{corollary}[theorem]{Corollary}
\theoremstyle{definition}

\theoremstyle{remark}
\newtheorem{remark}[theorem]{Remark}

\newtheorem{innercustomassumption}{}
\newenvironment{customassumption}[1]
{\renewcommand\theinnercustomassumption{#1}\innercustomassumption}
{\endinnercustomassumption}

\usepackage{pifont}
\newcommand{\xmark}{\text{\ding{55}}}

\def\bZ{\mathbf{Z}}

\def\D{\mathbf{D}}
\def\bQ{\mathbf{Q}}
\def\R{\mathbb{R}}

\def\E{\mathbb{E}}
\def\bF{\mathbf{F}}
\def\bE{\mathbf{E}}

\def\bP{\mathbf{P}}
\def\bR{\mathbf{R}}

\def\H{\mathbf{H}}
\def\I{\mathbf{I}}

\def\eps{\varepsilon}

\def\u{\mathbf{u}}

\def\x{\mathbf{x}}
\def\W{\mathbf{W}}
\def\A{\mathbf{A}}
\def\B{\mathbf{B}}

\def\tr{\textup{tr}}

\def\param{\boldsymbol{\theta}}
\def\Param{\boldsymbol{\Theta}}

\def\balpha{\boldsymbol{\alpha}}
\def\bbeta{\boldsymbol{\beta}}

\def\I{\mathbf{I}}
\def\r{\mathbf{r}}
\def\c{\mathbf{c}}

\DeclareMathOperator*{\argmin}{arg\,min}
\newcommand{\innp}[1]{\left\langle #1 \right\rangle}

\newcommand{\commHL}[1]{{\textcolor{blue}{#1}}}

\author{Danny Duan}
\address{Department of Mathematics, University of Wisconsin-Madison, Madison, Wisconsin, USA}
\email{bduan5@wisc.edu}
\author{Hanbaek Lyu}
\address{Department of Mathematics, University of Wisconsin-Madison, Madison, Wisconsin, USA}
\email{hlyu@math.wisc.edu}

\date{\today}

\title[Regularized Overestimated Newton]{Regularized Overestimated Newton}

\begin{document}

\begin{abstract}
We propose Regularized Overestimated Newton (RON), a Newton-type method with low per-iteration cost and strong global and local convergence guarantees for smooth convex optimization. RON interpolates between gradient descent and globally regularized Newton, with behavior determined by the largest Hessian overestimation error. Globally, when the optimality gap of the objective is large, RON achieves an accelerated  $O(n^{-2})$ convergence rate; when small, its rate becomes $O(n^{-1})$. Locally, RON converges superlinearly and linearly when the overestimation is exact and inexact, respectively, toward possibly non-isolated minima under the local Quadratic Growth (QG) condition. The linear rate is governed by an improved effective condition number depending on the overestimation error. Leveraging a recent randomized rank-$k$ Hessian approximation algorithm, we obtain a practical variant with  $O(\textup{dim}\cdot k^2)$ cost per iteration.  When the Hessian rank is uniformly below $k$, RON achieves a per-iteration cost comparable to that of first-order methods while retaining the superior convergence rates even in degenerate local landscapes. We validate our theoretical findings through experiments on entropic optimal transport and inverse problems.
\end{abstract}

\maketitle

\section{Introduction}
We study the following problem, where $f$ is convex and twice continuously differentiable:
\[
\param_{*} \in \argmin_{\param \in \R^d} f(\param).
\]
Newton’s method, with its rapid convergence and ability to adapt to local geometry, has remained a cornerstone in convex optimization for centuries \cite{conn2000trust} and has inspired the development of several influential methods, such as trust-region \cite{conn2000trust,grapiglia2015convergence}, 
 Levenberg-Marquardt (LM) regularized Newton   \cite{levenberg1944method,marquardt1963algorithm}, and cubic Newton \cite{nesterov2006cubic,nesterov2008accelerating}. Many important problems in machine learning and scientific computing (e.g., the Kantorovich dual of the entropic optimal transport \cite{cuturi2013sinkhorn} and overdetermined least squares \cite{gentle2012numerical}) have non-isolated local minima that form a submanifold. Both LM- and cubic-regularized Newton exhibit superlinear local convergence, under suitable assumptions, in possibly degenerate local landscapes with non-isolated minima \cite{yamashita2001rate, fan2005quadratic, li2004regularized, fischer2024levenberg, yue2019quadratic, rebjock2024fast}. Moreover, cubic Newton has an accelerated global convergence rate of $O(n^{-2})$, where $n$ denotes the iteration count. Recently, Mishchenko \cite{mishchenko2023regularized} proposed globally regularized Newton (GRN), an adaptive version of LM-Newton designed to mimic key properties of cubic Newton. It maintains a fast global convergence rate $O(n^{-2})$ and has superlinear local convergence in non-degenerate landscapes. But whether it retains such superior local convergence for non-isolated minima remains open. 

Despite this progress on Newton-type methods, they are not suitable for modern high-dimensional problems due to the prohibitively high computational cost of computing and inverting the Hessian, which in general incurs at least $O(d^{2})$ cost. To reduce this high computational cost, researchers have developed lightweight Newton-type methods that use some sort of Hessian approximation. 
Quasi-Newton methods \cite{broyden1970convergence,fletcher1970new,goldfarb1970family,shanno1970conditioning,liu1989limited} use cheap recursive Hessian approximations. Modern approaches based on random subspaces and Hessian sketching techniques have also emerged \cite{pilanci2017newton,gower2019rsn,jiang2024krylov,hanzely2020stochastic,hanzely2023sketch}.
However, despite some recent theoretical developments of quasi-Newton methods \cite{jin2024non,rodomanov2021rates}, a rate of convergence for general convex function seems to be unknown; and the sketching-based methods, unless operated in the trivial full-dimensional setting, behave similarly to first-order methods with $O(n^{-1})$ global convergence and linear local convergence only in non-degenerate landscapes. 

In this work, we task ourselves with developing a simple Newton-type method that operates efficiently in high dimensions but still maintains the superior convergence guarantees of Newton-type methods. Our method is based on the following simple and flexible framework for possibly randomized Newton-type methods that we call \textit{Regularized Overestimated Newton} (RON):
\begin{align}\label{eq:RON_high_level}
& \param_{n+1}\xleftarrow[]{}\param_{n} - ( \hat{\B}_{n} + \lambda_{n} \I )^{-1}  \nabla f(\param_{n}),\\
 &\text{where} \quad  \hat{\B}_{n}  \succeq \nabla^{2}f(\param_{n})\text{ and } \lambda_{n}=\sqrt{L_H\| \nabla f(\param_{n})\|}\nonumber.
\end{align}
Here $\hat{\B}_{n}$ is a (possibly randomized) overestimation of the true Hessian $\nabla^{2} f(\param_{n})$, meaning that $\hat{\B}_{n}-\nabla^{2} f(\param_{n})$ is positive semi-definite (PSD) almost surely, and $L_H$ is the Lipschitz constant of the Hessian of $f$. One may immediately notice that this is obtained from Mishchenko's GRN \cite{mishchenko2023regularized} by replacing the exact Hessian $\nabla^{2} f(\param_{n})$ with its overestimation $\hat{\B}_{n}$. In one extreme, taking $\hat{\B}_{n}= L \I $ where $L$ denotes the largest eigenvalue of the Hessian (assuming that it is finite), \eqref{eq:RON_high_level} becomes gradient descent with stepsize $(L+\lambda_{n})^{-1}$. One may use more geometry-aware overestimation (e.g., based on low-rank Hessian approximation) to strike a balance between the first-order method and GRN. 

\subsection{Contribution} 

We first provide global and local convergence analysis of RON \eqref{eq:RON_high_level} under minimal assumptions. A common theme is a phase transition from the Newton-type fast convergence into first-order-type convergence, where the phase transition point is determined by the competition between the optimality gap and the Hessian overestimation error. We also provide a practical implementation of RON using a randomized rank-$k$ Hessian approximation that has $O(dk^{2})$ per-iteration cost with numerical validation of our results. 

\begin{description}[style=unboxed,leftmargin=0.4cm]
    \item[$\bullet$] (\textit{Global convergence}) The expected function value gap decays at a rate $O(n^{-2})$ when itself is larger than the square of the mean Hessian overestimation error; afterwards, the rate becomes $O(n^{-1})$. (Thm. \ref{thm:convex})

    \item[$\bullet$] (\textit{Local convergence}) Assume the local landscape has curvature $\mu>0$ along non-degenerate directions (see $\mu$-{QG}\eqref{eq:local_landscape_assumptions}). When the gradient norm is sufficiently small depending on $\mu$ and the the worst-case Hessian overestimation error, we have superlinear local convergence under zero overestimation error (Thm. \ref{thm:local_unified}, Cor. \ref{cor:exact}) and linear under nonzero overestimation error (Thm. \ref{thm:local_unified}, Cor. \ref{cor:inexact}).

    \item[$\bullet$] (\textit{Practical implementation}) We provide a practical implementation of RON using a randomized rank-$k$ approximation of the Hessian, based on the randomly pivoted Cholesky \cite{chen2022randomly}. The resulting method operates with $O(dk^{2})$ per-iteration cost. It coincides with GRN \cite{mishchenko2023regularized} if $k$ is a uniform bound on the rank of the Hessian. In this case, our method has cheap $O(dk^{2})$ per-iteration cost, $O(n^{-2})$ global convergence, and superlinear local convergence under the relaxed $\mu$-{QG} condition. 

    \item[$\bullet$] (\textit{Numerical validation}) 
     We demonstrate superior numerical performance of RON for solving entropic optimal transport against the Sinkhorn algorithm, various Newton-type algorithms, and for solving large linear systems against state-of-the-art methods such as Randomized Kaczmarz \cite{strohmer2009randomized} and conjugate gradient \cite{hestenes1952methods}. 
\end{description}

\subsection{Related works}
${}$
\noindent \textbf{Globally regularized Newton.} Our work has been inspired by GRN proposed by Mishchenko \cite{mishchenko2023regularized} in 2023. This method is recovered from \eqref{eq:RON_high_level} by using the full Hessian $\hat{\B}_{n}=\nabla^{2}f(\param_{n})$. It was shown that GRN achieves $O(n^{-2})$ convergence rate, matching that of the Cubic Newton without solving a cubic subproblem, for minimizing a convex objective with Lipschitz continuous Hessian. GRN also enjoys local superlinear convergence when the objective is strongly convex. { More recently in \cite{doikov2024super}, the authors proposed a super-universal regularized newton that regularizes the Hessian by $\sigma_n\|\nabla f(\param_n)\|^{\alpha}\I$ with fixed $\alpha\in[2/3,1]$ and with $\lambda_n$ chosen adaptively in each iteration to satisfy the condition $\langle\nabla f(\param_{n+1}),\param_{n+1}-\param_{n}\rangle\ge \|\nabla f(\param_{n+1})\|^2/(4\sigma_n\|\nabla f(\param_n)\|^{\alpha}).$ They showed that this algorithm can self auto-adapt to a class of objectives whose smoothness is categorized by $q\in[2,4]$ and convexity categorized by $s\in[2,\infty].$ For general convex objective ($s=\infty$), the algorithm achieves global rate $O(n^{-(1-q)})$, which specializes to $O(n^{-2})$ when q=3, which corresponds to Lipschitz continuous Hessian. And for objectives with strongly convex and Lipschitz continuous Hessian, that is s=2, q=3, (or in general $2\le s<q$), the algorithm converges superlinearly.}

\noindent\textbf{Local convergence toward non-isolated minima.}
Many problems have non-isolated minima forming a submanifold, e.g., Kantorovich dual of entropic OT \cite{cuturi2013sinkhorn}, overdetermined least squares \cite{gentle2012numerical}, phase retrieval \cite{zhou2016geometrical}, blind deconvolution \cite{li2019rapid}, and linear residual networks \cite{hardt2016identity}. To address this, prior work extends Newton-type superlinear local convergence—usually under non-degeneracy—to possibly degenerate, non-isolated minima using LM regularization \cite{yamashita2001rate, fan2005quadratic, li2004regularized, fischer2024levenberg} or cubic regularization \cite{nesterov2006cubic,yue2019quadratic, rebjock2024fast}{, under local conditions like Quadratic Growth (QG) \cite{bonnans1995second}, Error Bound (EB) \cite{luo1993error}, or Polyak–Lojasiewicz (PL) \cite{polyak1963gradient}. (See Remark \ref{rmk:PL} for a detailed discussion.)} However, such guarantees remain unknown for GRN \cite{mishchenko2023regularized} and subspace-based methods \cite{pilanci2017newton, gower2019rsn, jiang2024krylov, hanzely2020stochastic, hanzely2023sketch}.

\noindent \textbf{Sketching-based Newton methods.}
Inspired by randomized linear algebra \cite{gower2017randomized}, many recent works proposed cheap Newton-type methods using random sketching.
Subspace Newton \cite{pilanci2017newton} achieves linear global convergence for self-concordant functions and linear local convergence for smooth and $\mu$-strongly convex functions. Their local rate is not applicable unless the objective is $\mu$-strongly convex—unlike our Cor. \ref{cor:inexact}, which only requires local PL.
Random Subspace Newton (RSN) \cite{gower2019rsn} guarantees linear local and $O(n^{-1})$ global convergence for strongly convex problems. Sketch-based cubic Newton variants such as Stochastic Subspace Cubic Newton (SSCN) \cite{hanzely2020stochastic}, Sketchy Global Newton (SGN) \cite{hanzely2023sketch}, and Krylov Subspace Cubic Regularized Newton (KCRN) \cite{jiang2024krylov} provide linear local and superlinear rates only when applied to the full space. In fact, sketch-and-project-based methods cannot attain superlinear convergence in subspace settings \cite{jiang2024krylov, hanzely2020stochastic, hanzely2023sketch}.

\noindent\textbf{Low-rank Hessian approximation.}
Doikov et al. \cite{doikov2024spectral} proposed a spectrally preconditioned GD similar to our RON \eqref{eq:RON_high_level}, using overestimated Hessians of the form $\hat{\B}_{n}=\llbracket \nabla^{2} f(\param_{n}) \rrbracket_{k}+\tau_{n}\I$ with exact rank-$k$ approximations. Their convex/local guarantees are weaker than ours but hold under stronger assumptions of strong convexity with quasi-self-concordance. Moreover, exact rank-$k$ decompositions are costly, and although power-method-based approximations are discussed, no theoretical guarantees are provided for the approximate case. They also provide nonconvex convergence analysis, but we do not pursue this direction in this work. 
    \begin{table*}[h!]
         \caption{Comparison between the proposed RON method and various benchmark methods.
        }
        \label{table:comparison}
        \vskip 0.15in
        \centering
        \scriptsize
        \begin{tabular}{|c|c|c|c|c|}
        \hline
        \multirow{2}{*} {\textit{Methods}} & \multicolumn{2}{c|}{\textit{Local convergence}} & \multirow{2}{*}{\textit{per-iter cost} } & \multirow{2}{*}{\textit{Convex global rate}} \\
        \cline{2-3}
        &  \textit{Property} & \textit{Rate}  &  &\\
        \hline
        GD \cite{nesterov2004introductory} & $\mu$-PL  \eqref{eq:local_landscape_assumptions} & Linear & $O(d)$ & $O(n^{-1})$\\
        \hline
         NAGD \cite{nesterov2004introductory} & $\mu$-strongly convex & Linear & $O(d)$ & $O(n^{-2})$\\
        \hline
        Newton \cite{kantorovich1948newton} &  $\mu$-strongly convex & quadratic & $O(d^3)$ & \xmark\\
        \hline
        BFGS \cite{rodomanov2021rates,rodomanov2021new,jin2024non}  &  $\mu$-stronly convex & superlinear & $O(d^2)$ & NA\\
        \hline
        L-BFGS \cite{liu1989limited,nocedal1980updating} & $\mu$-stronly convex & linear & $O(dm)$ & NA\\
        \hline
        RSN \cite{gower2019rsn} & $\begin{matrix}
            \text{relative}\\
            \mu\text{-strongly convex}
        \end{matrix}$  & Linear & $O(k^2d)$ & $O(n^{-1})$\\
        \hline
        SGN \cite{hanzely2023sketch} & 
        $\mathbf{g}_n\in\text{Range}(\H_n)$  & Linear & $O(k^2d)$ & $O(n^{-2})$\\
        \hline
         \text{CRN} \cite{nesterov2006cubic}
       & $\mu$-PL \eqref{eq:local_landscape_assumptions} &  quadratic & $O(d^3)$ &$O(n^{-2})$ \\
        \hline
        \text{GRN} \cite{mishchenko2023regularized}
       & $\mu$-strongly convex &  superlinear & $O(d^3)$ &$O(n^{-2})$ \\
        \hline
        \text{SSCN} \cite{hanzely2020stochastic}
         & $\mu$-strongly convex &  linear & $O(k^2d)$ &$O\left(\frac{d-k}{k}\frac{1}{n}+\frac{d^2}{k^2}\frac{1}{n^2}\right)$ \\
        \hline
        \text{KCRN} \cite{jiang2024krylov}
         & $\mu$-strongly convex  &  linear & $O(kd)$ &$O\left(\frac{1}{kn}+\frac{1}{n^2}\right)$ \\
        \hline 
        \textbf{RON} (ours)
         & $\mu$-QG\eqref{eq:local_landscape_assumptions} & $\begin{matrix}
            \text{linear}\\
            \text{superlinear if $\overline{\eps}=0$}
        \end{matrix}$ & $O(k^2d)$ &$O\left(\frac{1}{n^2}\right)\to O\left(\frac{1}{n}\right)$ \\
        \hline
        \end{tabular}
    \end{table*}
    
\subsection{Notation}
We use $\lVert \cdot \rVert_{F}$ and $\|\cdot\|_{2}$ to denote the matrix Frobenius norm and matrix 2-norm, respectively. For vectors, we use $\|\cdot\|$ to denote the Euclidean norm. 
A square matrix $\A\in \R^{d\times d}$ is \textit{positive semi-definite} (PSD) if $\x^{T} A \x\ge 0$ for all $\x\in \R^{d}$ and is \textit{positive definite} (PD) if $\x^{T} A \x>0$ for all nonzero $\x\in \R^{d}$. For a real symmetric matrix $\A$, let $\lambda_{\min}(\A)$ and $\lambda_{\max}(\A)$ denote the minimum and the maximum eigenvalues of $\A$. A real symmetric matrix $\A$ is PSD if and only if $\lambda_{\min}(\A)\ge 0$, and PD if and only if $\lambda_{\min}(\A)> 0$. The \textit{trace norm} of $\A$ is $\tr(\A)$, which equals the sum of all eigenvalues of $\A$. For two square matrices $\A,\B\in \R^{d\times d}$, denote $\A\succeq \B$ if $\A-\B$ is PSD.

\subsection{Organization} 
The paper is organized as follows. An overview of our algorithm description is in \cref{sec:alg_overview}, the main results can be found in \cref{sec:main}. The proofs of global and local convergence results are established in \cref{sec:global_pf} and \cref{sec:local_pf}, respectively. We offer a detailed implementation of RON in \cref{sec:alg_detail} and some numerical experiments in \cref{sec:experiments}.

\section{RON with low-rank Hessian estimation}\label{sec:alg_overview}

We propose to use RON \eqref{eq:RON_high_level} with the Hessian overestimation based on a low-rank approximation. For a fixed rank parameter $k$, consider RON based on rank-$k$ approximation of the Hessian that we call the \textit{rank-$k$ RON}:  
\begin{align}\label{eq:RPC_hessian_approximation}
 & \textup{\eqref{eq:RON_high_level} with }  \hat{\B}_{n}  \leftarrow \hat{\H}_{n} + \rho_{n} \I, 
\\
&\hspace{1cm} \textup{where $\hat{\H}_{n}$}  \approx \llbracket \nabla^{2} f(\param_{n})  \rrbracket_{k}, 
\quad \rho_{n} \ge \| \nabla^{2}f(\param_{n}) - \hat{\H}_{n} \|_{2}. \nonumber 
\end{align}
Here $\hat{\H}_{k}$ has rank $k$ and  $\llbracket \A \rrbracket_k$ denotes the best  rank-$k$ approximation of $\A$ (i.e., $\llbracket \A \rrbracket_k=\sum_{i=1}^{k} \sigma_{i}\u_{i} \u_{i}^{\top}$ where $\sigma_{1}\ge\sigma_{2}\ge\cdots$ are eigenvalues of $\A$ and $\u_{1},\u_{2},\cdots$ are the corresponding unit eigenvectors). The above provides a flexible framework where a user can choose her favorite low-rank PSD matrix approximation method from the vast literature (see Musco and Woodruff \cite{musco2017sublinear} and references therein) to compute $\hat{\H}_{n}$. Examples include randomized $k$-SVD \cite{halko2011finding, tzeng2022improved}, power method \cite{golub2013matrix}, and Lanczos method \cite{lanczos1950iteration}.

We propose a particular variant of \eqref{eq:RPC_hessian_approximation} that computes $\hat{\H}_{n}$ using \textit{randomly pivoted Cholesky} (RPC) by Chen, Epperly, Tropp, and Webber \cite{chen2022randomly}, which has many pleasing properties in our context of RON. Given a PSD matrix $\A$ and a rank parameter $k$, RPC outputs a $k$-column Nystr\"{o}m approximation 
\cite[Sec. 19.2]{martinsson2020randomized} $\hat{\A}$ in the factor form $\bF\bF^{\top}$, $\bF\in \R^{d\times k}$. This is a randomized algorithm since the $k$ columns of $\A$ are sampled with probabilities proportional to the corresponding diagonal entries of $\A$. It is always an \textit{underestimation} $\mathbf{O}\preceq \hat{\A}\preceq \A$, so $\| \A-\hat{\A} \|_{2}\le \tr(\A-\hat{\A})$. The trace norm bound can be easily computed. Thus we can specialize the rank-$k$ RON as 
\begin{align}\label{eq:RON_RPC}
&\textup{\eqref{eq:RPC_hessian_approximation} with   }\hat{\H}_{n}\leftarrow \bF_{n}\bF_{n}^{\top} \approx \llbracket \nabla^{2} f(\param_{n})  \rrbracket_{k} \,\,  \textup{using RPC \cite{chen2022randomly}},\\
&\hspace{1cm}  \rho_{n}\leftarrow \tr\left( \nabla^{2} f(\param_{n}) - \hat{\H}_{n} \right). \nonumber
\end{align}
The above procedure has a cheap per-iteration cost of $O(d k^{2})$; this is the cost of running RPC with $k$ columns, and the inversion \eqref{eq:RON_high_level} also takes the same amount of computation using Woodbury inversion \cite{woodbury1950inverting}. When $k=O(1)$, this method has computational complexity comparable to the first-order methods.

Another important advantage of the implementation \eqref{eq:RON_RPC} is that the overestimation error becomes zero if $k\ge \textup{rank}(\nabla^{2} f(\param_{n}))$. In this case, we have exact recovery of the Hessian $\hat{\H}_{n}=\nabla^{2} f(\param_{n})$ (so $\rho_{n}=0$) almost surely (see Cor. \ref{cor:RPC_exact}).  Therefore, when $k$ is a uniform upper bound on the rank of the Hessian, our method \eqref{eq:RON_RPC} coincides with GRN \cite{mishchenko2023regularized} with a cheap per-iteration cost $O(dk^{2})$.

\section{Main results}
\label{sec:main}
We provide global and local convergence analysis of our RON algorithm \ref{algorithm:RON}. For both analyses, we assume that the objective has Lipschitz continuous gradient and Hessian as below:

\begin{customassumption}{A1}[Smooth objective]\label{assumption:A1}
    The convex objective $f:\R^{d}\to \R$  satisfies  $f^{*}:=\inf f>-\infty$ and is twice continuously differentiable, and for some constants $L,L_H>0$, 
    \begin{align}
        \|\nabla f(\param)-\nabla f(\param')\|_2 \leq L\|\param-\param'\|, \quad 
         \|\nabla^2 f(\param)-\nabla^2 f(\param')\| \leq L_H\|\param-\param'\| \label{eq:assump2}
    \end{align}
    for all $\param,\param'\in\mathcal{L}:=\{\param:f(\param)\leq f(\param_0)\}.$ Furthermore, $\mathcal{L}$ is compact with diameter $D$. 
\end{customassumption}

Recall that for each $\param\in\mathcal{L}$, RON \eqref{eq:RON_high_level} calls for a (possibly randomized) overestimation $\hat{\B}(\param)\succcurlyeq \nabla^2f(\param)$. Define the expected and almost sure upper bound on the Hessian overestimation error within the level set $\mathcal{L}$ as below:
\begin{align}\label{eq:def_worst_case_overestimation_errors}
\overline{\eps}_{\E}:=\sup_{\param\in \mathcal{L}}\E_{\omega}[\|\hat\B(\param;\omega)-\nabla^2 f(\param)\|_{2}], \quad \overline{\eps}:=\sup_{\omega\in \Omega}\sup_{\param\in \mathcal{L}} \|\hat\B(\param; \omega)-\nabla^2 f(\param)\|_{2},
\end{align}
where $\omega\in \Omega$ denotes the randomness in the Hessian approximation oracle $\hat{\B}$. For any iterate $\param_n$ produced by the algorithm, also denote $\eps_n=\|\hat{\B}_n-\nabla^2f(\param_n)\|_2$. Notice that $\overline{\eps}\le L$ with the trivial Hessian overestimation $\hat{\B}_{n}=L\I$. Also, if we use $\hat{\B}_{n}=\llbracket \nabla^{2} f(\param_{n})\rrbracket_{k}+\sigma_{k+1}(\nabla^{2} f(\param_{n}))\I$, then $\overline{\eps}\le \sup_{\param\in \mathcal{L}} \sigma_{k+1}(\nabla^{2} f(\param))$.   

{
\begin{lemma}
    If $\hat{\B}_n=\hat{\H}_n+\rho_n\I$ as in \eqref{eq:RPC_hessian_approximation}. Then for any $\tau_n\ge \eps_n$, we have
    $\sigma_{\min}(\hat{\B}_n)+\tau_n\ge \frac{\eps_n+\tau_n}{2}$.
    If $\hat{\H}_n$ is in fact a PSD underestimation of the Hessian, namely $\mathbf{O}\preceq \H_n\preceq \nabla^2f(\param_n)$, then $\sigma_{\min}(\hat{\B}_n)\ge \eps_n.$    
\end{lemma}
\begin{proof}
    The first inequality is equivalent to $\sigma_{\min}(\hat{\B}_n)\ge (\eps_n-\tau_n)/2$, which is true since $\tau_n\geq\eps_n$ by assumption and $\sigma_{\min}(\hat{\B}_n)\geq 0.$ For the second claim, if $\mathbf{O}\preceq \H_n\preceq \nabla^2f(\param_n)\preceq \hat{\B}_n$, then
    \[
    \eps_n=\lambda_{\max}(\hat{B}_n-\nabla^2f(\param_n))=\lambda_{\max}(\hat{\H}_n-\nabla^2f(\param_n)+\rho_n\I)\le \rho_n\le \sigma_{\min}(\hat{\B}_n).
    \]
\end{proof}
This shows that for a reasonable formation of the Hessian overestimation $\hat{\B}_n$, we may further assume that almost surely $\sigma_{\min}(\hat{\B}_n)\ge c\eps_n$ for some $c>0$ independent of $n.$ Note that this requirement holds for $c=1$ if $\hat{\B}_n$ can be rewritten as $\hat{\H}_n+\rho_n\mathbf{I}$ where $\hat{\H}_n$ is a PSD underestimation of the Hessian. In general, we can always make sure this is true with $c=1/2$ by updating the initial overestimation by adding $\tau_n\mathbf{I}$ with $\tau_n\geq \eps_n$. An easily computable option could be  $\tau_{n}=\tr(\hat{\B}_{n}-\nabla^{2}f(\param_{n}))$.
}

The following result establishes the two-phase global convergence rate of RON. 

\begin{theorem}[Global convergence for convex objective]\label{thm:convex}
	Assume $f$ satisfies \cref{assumption:A1}. Suppose $(\param_n)_{n\geq 0}$ is generated by RON \eqref{eq:RON_high_level} with arbitrary $\param_0$. 
    Assume $\sigma_{\min}(\hat{\B}_{n})\ge c\eps_{n}$ for some $c>0$ and for all $n\ge 1$. Denote $A_{n}:=f(\param_{n})-f^{*}$ and $T:=\inf \{ n\ge 0 \,|\,\E[A_n] \leq  (D\overline{\eps}_{\E})^{2}/L_H  \}$. Then 
	 $\E[A_{n}] \le C_1n^{-2}$ for $0\le n < T$ and $\E[A_{n}]\le C_2 \overline{\eps}_{\E}n^{-1}$  for $n\ge T$ where $C_1=O(D^3L_H), C_2=O(\max\{D^3L_H,D^2\})$.
\end{theorem}

Note that we have not tried to optimize the constants in Thm. \ref{thm:convex}. 

Next, we turn our attention to the local convergence of RON. We allow the possibility that the minima of $f$ may not be isolated. So, instead of having one global minimizer $\param_*$, we denote a set of global minima $\mathcal{S}:= \argmin_{\param} f(\param)$. Specifically, $\mathcal{S}$ is a singleton if $f$ is strictly convex, but in the degenerate case where the Hessian of $f$ near every global minimizer $\param_{*}\in \mathcal{S}$ has rank $r<d$, 
$\mathcal{S}$ is a $(d-r)$-dimensional submanifold by constant rank theorem \cite{john2012introduction}. The algorithm can now converge towards any arbitrary minimizer $\param_{*}\in \mathcal{S}$ instead of a singleton. This additional degree of freedom makes it highly non-trivial to establish local convergence results for the general situation with non-isolated local minima.

In order to control such a situation, we further impose the following geometric assumption of local Quadratic Growth (QG) condition:

\begin{customassumption}{A2}[Local QG condition]\label{assumption:A2}
    There exists a constant $\mu>0$ and $\eta>0$ such that $f$ is \textit{$\mu$-QG} near every minimizer $\param_{*}\in \mathcal{S}$, that is, 
\begin{align*}
   f(\param)-f(\param_{*})\ge \frac{\mu}{2}\text{dist}(\param,\mathcal{S})^2  \quad \textup{for all $\param$ with $\| \nabla f(\param) \|\le \eta$}.
\end{align*} 
\end{customassumption}
The above assumption is strictly weaker than local $\mu$-strong convexity of $f$ near the isolated minimizer. Further, $\mu$-QG is weaker than $\mu$-EB and $\mu$-PL, which are some other commonly used assumptions on non-isolated local minima (See Remark \ref{rmk:PL} for a detailed discussion). $\mu$-PL (thus $\mu$-QG) is known to be satisfied by a number of problems including phase retrieval \cite{zhou2016geometrical}, blind deconvolution \cite{li2019rapid}, and linear residual networks \cite{hardt2016identity}. Moreover, $\mu$-PL has been used in recent works on cubic Newton in degenerate landscapes 
\cite{nesterov2006cubic,rebjock2024fast}. And $\mu$-EB was used in LM-regularized Newton in degenerate landscapes \cite{yamashita2001rate, fan2005quadratic, li2004regularized, fischer2024levenberg}. 

The following is our key result for local convergence of RON in general $\mu$-QG local landscapes. It will be helpful to introduce the following random time determined by the size of the gradient norm. For each $\gamma>0$, define 
\begin{align}\label{eq:N_gamma}
N(\gamma):=\inf\left\{n\ge 0\,:\, \|\nabla f(\param_{n})\|< \min\{ (\gamma\mu)^2/L_H, \eta \} \right\},
\end{align}
where $\mu,\eta$ are as in \ref{assumption:A2}. Accordingly, we define: 
\begin{align}\label{eq:Ni}
    N_{1} &:=N(\gamma_{1}) \quad \textup{where} \quad \gamma_{1}:=\min\left\{0.2,\frac{20}{33}\frac{\mu}{1.2\mu+\overline{\eps}}\right\}. 
\end{align}
We can view the random stopping time $N_{1}$ as the time at which local convergence begins.

\begin{theorem}[Local contraction of gradient norm]\label{thm:local_unified}
Suppose $f$ satisfies \ref{assumption:A1} and \ref{assumption:A2}.  
Let $N_{1}$ be as in \eqref{eq:Ni}. 
Denote $\overline{\lambda}_{n}:=\eps_n+\sqrt{L_{H}\| \nabla f(\param_{n}) \|}$ with $\eps_{n}:=\| \hat{\B}_{n}-\nabla^{2}f(\param_{n}) \|_{2}$. 
On the event $n\ge N_{1}$, almost surely, $\|\nabla f(\param_{n+1})\|\leq \|\nabla f(\param_n)\|$ and more precisely 
\begin{align}\label{eq:grad_contraction}
     \|\nabla f(\param_{n+1})\|\leq \frac{\overline{\lambda_n}}{\mu+\overline{\lambda_n}}\|\nabla f(\param_n)\|+\frac{5}{4}\frac{\sqrt{L_H}}{\mu}\frac{\overline{\lambda_n}}{\mu+\overline{\lambda_n}}\|\nabla f(\param_n)\|^{3/2}+ \frac{2L_H}{\mu^2}\|\nabla f(\param_n)\|^2.
\end{align}
\end{theorem}

We may specialize this result depending on whether the worst-case Hessian overestimation error $\overline{\eps}$ is zero or not. First, suppose $\overline{\eps}=0$, then $\eps_n=0$ for all $n\geq 0$. Thus $\gamma_1=0.2,$ and $\mu+\overline{\lambda_n}\leq 1.2\mu.$ 
Then if $n\geq N_1$, Thm. \ref{thm:local_unified} gives the following  gradient norm contraction
\begin{align}\label{eq:grad_contraction2}
\|\nabla f(\param_{n+1})\|\leq \frac{25}{24}\frac{\sqrt{L_H}}{\mu}\|\nabla f(\param_n)\|^{3/2}+\left(\frac{25}{24}+2 \right)\frac{L_H}{\mu^2}\|\nabla f(\param_n)\|^2.
\end{align}
Because of our choice of $n\ge N_{1}$, the quadratic term in the RHS above is dominated by the first term, so we get a super-linear convergence rate as stated below. 



\begin{corollary}[Superlinear local convergence with exact Hessian]\label{cor:exact}
Keep the same setting as in Thm. \ref{thm:local_unified} and further assume $\bar{\eps}=0$. 
Let $N_{1}$ be as in \eqref{eq:Ni}.
Then almost surely for any $n\geq  N_{1}$, 
\begin{align*}
    \|\nabla f(\param_{n+1})\|\leq \frac{25}{12}\frac{\sqrt{L_H}}{\mu}\|\nabla f(\param_n)\|^{3/2}. 
\end{align*}
In particular, we have superlinear local convergence for $n$ sufficiently large. 
\end{corollary}

We remark that when $\overline{\eps}=0$ our algorithm coincides with GRN \cite{mishchenko2023regularized}. Hence Cor. \ref{cor:exact} extends the superlinear convergence of GRN under $\mu$-strong convexity in \cite[Thm. 2]{mishchenko2023regularized}.

Next, we consider the case when the Hessian overestimation error is nonzero. In this case, we have a local linear convergence with rate given by the following ``effective condition number'' $\frac{1.2\mu + \overline{\eps}_{\E}}{\mu}$. 

\begin{corollary}[Linear local convergence with overestimated Hessian]\label{cor:inexact}
Keep the same setting as in Thm. \ref{thm:local_unified} and further assume $\bar{\eps}\in(0,\infty)$. 
Let $N_{1}$ be as in \eqref{eq:Ni}. Fix any $\delta \in (0,1)$ and let $N_{2}=N(\gamma_{1}\delta)$. Then  for all $n>N_{2}$, we have linear convergence in expectation: For all $m\ge 1$,
\begin{align*}
    \E[\|\nabla f(\param_{n+m})\| \,|\, n>N_{2}]&\leq \left(1-(1-\delta)\frac{\mu}{1.2\mu+\overline{\eps}_{\E}}\right)^{m}\E[\|\nabla f(\param_n)\| \,|\, n>N_{2}]. 
\end{align*}
\end{corollary}

Corollary \ref{cor:inexact} can also be deduced from the main local convergence result (Thm. \ref{thm:local_unified}) without much difficulty. 
From \eqref{eq:grad_contraction} one can deduce that whenever $n>N_{2}$, 
\begin{align*}
\|\nabla f(\param_{n+1})\| &\le \frac{(\delta+0.2)\mu+\eps_n}{1.2\mu+\eps_n}\|\nabla f(\param_n)\|. 
\end{align*}
Then, taking the conditional expectation on the event $n> N_2$ using Jensen's inequality gives the desired result. 
{
    
One can compare this rate with that of Nesterov's accelerated gradient descent (NAGD) and gradient descent (GD). For $\mu$-strongly convex and $L$-smooth objectives with condition number $\kappa=L/\mu$, NAGD and GD have local convergence rates $(1-1/\sqrt{\kappa})^n$ and $(1-1/\kappa)^n$ respectively. These results are known under weaker assumptions, such as the quadratic growth condition \cite{necoara2019linear}. These linear rates can deteriorate significantly for ill-conditioned problems, and the benefit of cheap iterations is often not sufficient to compensate for such weakness. Our RON has a superior local convergence rate as we demonstrate in Thm. \ref{thm:local_unified}. With exact Hessian overestimation, it has superlinear convergence under $\mu$-QG condition. With inexact Hessian overestimation, it has linear convergence, our Cor. \ref{cor:inexact} (after recasting the gradient norm to optimality gap) yields local convergence rate $E[f(\param_n)-f^*]=O((1-\frac{\mu}{1.2\mu+\bar{\eps}_{E}})^{2n}).$ This indicates that our local linear rate is superior to that of (NA)GD when $\bar{\eps}_{E}\ll L-0.2\mu$. In particular, if the overestimation error $\bar{\eps}_E\le C\mu$ for some constant $C$, then our local convergence rate bound is independent of the condition number. These advantages in local convergence rates come with a relatively cheap per-iteration cost for a low-rank local landscape. RON also allows the use of randomized Hessian approximation oracles. 
}

\section{Global convergence for convex objectives}
\label{sec:global_pf}
In this section we will prove Theorem \ref{thm:convex}. Following the literature on cubic regularized Newton \cite{nesterov2006cubic} and globally regularized Newton \cite{mishchenko2023regularized}, we denote 
\begin{align*}
	r_{n}:= \|\param_{n+1}-\param_n\|.
\end{align*}
We will also denote the (possibly randomized) overestimation error of the Hessian of the objective evaluated at the $n$-th parameter $\param_{n}$ and its spectral norm as 
\begin{align*}
	\bE_{n}:=\hat{\B}_n-\nabla^{2} f(\param_{n}), \quad \eps_n=\|\bE_{n}\|_2,
\end{align*}
notice that $\bE_n$ is PSD by \eqref{eq:RON_high_level}. We will deduce several lemmas for the Newton-type update \eqref{eq:RON_high_level}. Throughout this section we will also denote
\[
\overline{\lambda}_n:=\lambda_n+\eps_n
\]


\begin{lemma}[Stability]\label{lem:stability}
	Suppose \ref{assumption:A1} holds. 
	Fix $n\ge 0$, $\param_{n}$, and let $\param_{n+1}$ be as in \eqref{eq:RON_high_level}. Then almost surely,
    \begin{align}
        L_{H}
        r_{n} &\le \frac{ \lambda_{n}^{2}}{\lambda_{n}+\sigma_{\min}(\hat{\B}_{n}) },  
        \label{eq:stability1} \\
        \|\nabla f(\param_{n+1})\| & \leq \frac{3}{2}\overline{\lambda}_n r_{n}\leq  \frac{3}{2}\frac{\overline{\lambda}_n}{\lambda_n+\sigma_{\min}({\hat{\B}_n})}\|\nabla f(\param_n)\|\label{eq:stability2}.
    \end{align}
\end{lemma}
\begin{proof}
    The second inequality in \eqref{eq:stability2} follows from $\overline{\lambda}_{n}\ge 0$ and 
	\begin{align*}
		r_n=\|\param_{n+1}-\param_{n}\| =\|  (\hat{\B}_{n}+\lambda_{n}\I)^{-1}\nabla f(\param_n)\| 
		\leq \frac{1}{\lambda_{n}+\sigma_{\min}(\hat{\B}_n)}\|\nabla f(\param_n)\|.
	\end{align*}
    
    Since $\lambda_n^{2}= L_H\|\nabla f(\param_n)\|$, continuing the calculation above, 
    \[
    r_n
		\leq \frac{1}{\lambda_{n}+\sigma_{\min}(\hat{\B}_n)}\|\nabla f(\param_n)\| =  \frac{\lambda_n^2}{L_H(\lambda_{n}+\sigma_{\min}(\hat{\B}_n))},
    \]
     which proves \eqref{eq:stability1}. 
	Next, rearranging the update in \eqref{eq:RON_high_level}, we can write 
	\begin{align*}
		\nabla f(\param_{n})=-\lambda_{n} (\param_{n+1}-\param_{n})-\hat{\B}_{n} (\param_{n+1}-\param_{n}) . 
	\end{align*}
Using this, note that 
\begin{align*}
	&\hspace{-0.5cm}\|\nabla f(\param_{n+1})\|\\
    &= \|\nabla f(\param_{n+1})-\nabla f(\param_{n})+\nabla f(\param_{n})\|\\
    &= \|\nabla f(\param_{n+1}) - \nabla f(\param_n) - \hat{\B}_{n}(\param_{n+1} - \param_n) - \lambda_n (\param_{n+1} - \param_n) \| \notag \\
	&\leq \|\nabla f(\param_{n+1}) - \nabla f(\param_n) - \nabla^2f(\param_n)(\param_{n+1} - \param_n)\| +  \|(\bE_n+\lambda_n\mathbf{I})(\param_{n+1} - \param_n)\| \notag \\
	&\leq \frac{L_{H}}{2} r_{n}^{2} +\overline{\lambda}_{n} r_{n} \le \left(\frac{1}{2}+1\right)\overline{\lambda}_nr_{n},
\end{align*}
where the last inequality uses \eqref{eq:stability1} and $\lambda_n\leq \overline{\lambda}_n$. 
\end{proof}

Next, we prove a lower bound for per-iteration improvement in objective value. 

\begin{lemma}[Descent]\label{lem:descent}  
	Under the same setting as in Lem. \ref{lem:stability}, almost surely,
	\begin{align}\label{eq:perstep_impv}
		f(\param_n)-f(\param_{n+1})\geq 		
                \frac{1}{2}(\lambda_n+\sigma_{\min}(\hat{\B}_{n})) r_{n}^{2} 
	\end{align}
\end{lemma}
\begin{proof}    
    By the $L_{H}$-Lipschitz continuity of the Hessian, 
    (see \cite{nesterov2006cubic}), we have
	\begin{align*}
		&f(\param_n)-f(\param_{n+1}) \\
        &\quad \geq -\innp{\nabla f(\param_n),\param_{n+1}-\param_{n}}-\frac{1}{2}\innp{\nabla^2f(\param_n)(\param_{n+1}-\param_n),\param_{n+1}-\param_n}-\frac{L_H}{6}r_n^3.
	\end{align*}
   Recall we can write $-\nabla f(\param_n)=(\hat{\B}_n+\lambda_n\mathbf{I})(\param_{n+1}-\param_n)$ and $\hat{\B}_n=\nabla^{2}f(\param_{n})+\bE_{n}$, then
	\begin{align*}
		f(\param_n)-f(\param_{n+1}) &\geq \innp{(\hat{\B}_n+\lambda_n\mathbf{I})(\param_{n+1}-\param_n),\param_{n+1}-\param_{n}} \\
        & \hspace{2cm} -\frac{1}{2}\innp{\nabla^2f(\param_n)(\param_{n+1}-\param_n),\param_{n+1}-\param_n}-\frac{L_H}{6}r_n^3\\
		&=\frac{1}{2}\innp{ \underbrace{(\hat{\B}_n+\lambda_n\mathbf{I})}_{\succeq {(\lambda_n+\sigma_{\min}({\hat{\B}_{n}}))}\mathbf{I}}(\param_{n+1}-\param_n),\param_{n+1}-\param_{n}} \\
        &\hspace{2cm} +\frac{1}{2}\innp{ \underbrace{(\lambda_{n}\I +\bE_{n}}_{\succeq {\lambda_n}\mathbf{I}} )(\param_{n+1}-\param_n),\param_{n+1}-\param_n}-\frac{L_H}{6}r_n^3\\
		&\geq 
        \frac{1}{2}(\lambda_n+\sigma_{\min}({\hat{\B}_{n}})) r_n^2+ 
       \frac{1}{2}\lambda_n r_n^2-\frac{L_H}{6}r_n^3 
	\end{align*}
    Then notice that by \eqref{eq:stability1}, $\frac{L_H}{6}r_n^3\leq \frac{1}{6}r_n^2\frac{\lambda_n^2}{\lambda_n+\sigma_{\min}(\hat{\B_n})}\leq\frac{1}{6}\lambda_nr_n^2$. 
    Hence  continuing the calculation we have
	\begin{align*}
	f(\param_n)-f(\param_{n+1}) &\geq 
        \frac{1}{2}(\lambda_n+\sigma_{\min}({\hat{\B}_{n}})) r_n^2+ 
       \frac{1}{2}\lambda_n r_n^2-\frac{L_H}{6}r_n^3\\
       &\geq \frac{1}{2}(\lambda_n+\sigma_{\min}({\hat{\B}_{n}})) r_n^2.
	\end{align*}
\end{proof}

We are now ready to establish two-phase complexity of our algorithm stated in Thm. \ref{thm:convex}. Our proof follows the approach in \cite{mishchenko2023regularized}. 

\begin{proof}[\textbf{Proof of Theorem \ref{thm:convex}}]
	
	We also denote the objective value gap $A_{n}:= f(\param_{n}) - f^{*}$. By the assumption \ref{assumption:A1}, $f^{*} = \inf f$ is attained at some $\param_{*}$. 
	By Lemma \ref{lem:descent}, we have $f(\param_{n})\le \cdots \le f(\param_{0})$.  Also by the assumption \ref{assumption:A1} $\mathcal{L}=\{\param:f(\param)\leq f(\param_0)\}$ is compact    
    with $\| \param_{n} - \param_{*}\|\le D$  for all $n\ge 0$. Then by the convexity of $f$, \begin{align}\label{eq:objective_gap_upper_bd_convex}
		A_n=f(\param_{n}) - f^{*} \le \langle \nabla f(\param_{n}),\, \param_{n}-\param_{*} \rangle \le \| \nabla f(\param_{n}) \| \cdot \| \param_{n}-\param_{*} \| \le D \| \nabla f(\param_{n}) \| . 
	\end{align}
    Recall \eqref{eq:stability2}, and that $\sigma_{\min}(\hat{\B}_n)\geq c\eps_n$, we have
    \begin{align}\label{eq:stability2_simplified}
    \|\nabla f(\param_{n+1})\|\leq \frac{3}{2}\frac{\lambda_n+\eps_n}{\lambda_n+\sigma_{\min}(\hat{\B}_n)}\|\nabla f(\param_n))\|\le \frac{3}{2(1\wedge c)}\|\nabla f(\param_n)\|
    \end{align}
	Following \cite{mishchenko2023regularized}, define sets of indices 
	\begin{align*}
		\mathcal{I}_{\infty}:= \left\{k\ge 0\,|\,  \E[\| \nabla f(\param_{k+1})\|]  \ge \frac{(1\wedge c)}{3}\E[\| \nabla f(\param_{k})\|] \right\}, \qquad 	\mathcal{I}_{k}:= \{i\in \mathcal{I}_{\infty}\,|\, i\le k \}. 
	\end{align*} 

   Fix any fixed $n\in \mathcal{I}_{\infty}$. By Lemma \ref{lem:stability}, 
	\begin{align*}
		\frac{(1\wedge c)}{3} \|\nabla f(\param_{n})  \| \overset{n\in \mathcal{I}_{\infty}}{\le} \|\nabla f(\param_{n+1})  \|  &\overset{\eqref{eq:stability2}}{\le}  \frac{3}{2}\overline{\lambda}_n r_n
	\end{align*}
	Hence $r_{n} \ge \frac{2(1\wedge c)}{9} \frac{\| \nabla f(\param_{n}) \|}{\overline{\lambda}_n}$.
	Moreover, by Lemma \ref{lem:descent}, denoting $B=\frac{1}{2}\left(\frac{2}{9}\right)^2(1\wedge c)^3$, 
    \begin{align*}
			A_{n+1}-A_{n}=f(\param_{n+1}) - f(\param_{n}) 
            &\overset{\eqref{eq:perstep_impv}}{\le} 	-\frac{1}{2} (\lambda_{n}+\sigma_{\min}(\hat{\B}_n)) r_{n}^{2} \\
			&\le - \frac{1}{2} \left(\frac{2(1\wedge c)}{9}\right)^2 \frac{\lambda_{n}+\sigma_{\min}(\hat{\B}_n)}{\overline{\lambda_n}^2}\|\nabla f(\param_n)\|^2\\
            &\overset{(a)}{\le}  -B \frac{\|f(\param_n)\|^2}{\lambda_n+\eps_n} 
            \overset{(b)}{\leq} -B  \frac{(A_{n}/D)^{2}}
            {\sqrt{L_H A_{n}/D}+\eps_n}\\
            &=-BD^{-3/2}\frac{A_n^2}{\sqrt{L_HA_n}+\eps_n\sqrt{D}},
	\end{align*}    
      where (a) uses 
      $\frac{\lambda_n+\sigma_{\min}(\hat{B}_n)}{\overline{\lambda_n}}\geq\frac{\lambda_n+c\eps_n}{\lambda_n+\eps_n}\geq 1\wedge c$ because $\sigma_{\min}(\hat{\B}_n)\ge c\eps_{n}$,
      and (b) 
    follows from \eqref{eq:objective_gap_upper_bd_convex} and the fact that $x\to \frac{x^2}{a+b\sqrt{x}}$ is increasing for all $a,b>0$.


   Recall $\mathcal{F}_{n}$ is the $\sigma$-algebra generated by the information up to iteration $n$, in particular, $A_n$ is measurable w.r.t. $\mathcal{F}_{n}$. First taking the conditional expectation given $\mathcal{F}_{n}$, using Jensen's inequality with $\E[\eps_{n}|\mathcal{F}_n]\le \overline{\eps}_{\E}$, we get 
   \begin{align*}
       \E[A_{n+1}|\mathcal{F}_n]-A_n&\leq -BD^{-3/2}\frac{A_n^2}{\sqrt{D}\E[\eps_n|\mathcal{F}_n]+\sqrt{L_HA_n}}\\
       &\le -BD^{-3/2}\frac{A_n^2}{\sqrt{D}\overline{\eps}_{\E}+\sqrt{L_HA_n}}
   \end{align*}
   {\color{orange}}
   Then take the total expectation and use Jensen's inequality again to get 
	\begin{align*}
	\E[A_{n+1}]-\E[A_{n}]  \le  -BD^{-3/2}\frac{\E[A_{n}]^{2}}{\sqrt{D}\overline{\eps}_{\E} +\sqrt{L_{H}}  \, \E[A_{n}]^{1/2}}=-\frac{\E[A_n]^2}{c_1+c_2\E[A_n]^{1/2}},
	\end{align*}
    where we have denoted $ c_1=D^2\overline{\eps}_{\E}/B,
         c_2=D^{3/2}\sqrt{L_H}/B.$
    We will proceed by considering two cases depending on which of the two terms in the denominator above dominates the other. To which end, define
    \[
    T:=\inf_{i\geq 0 }\left\{i:\E[A_i]\leq (c_1/c_2)^2=D\overline{\eps}_{\E}^2/(L_H)\right\}.
    \]
    Recall that by \eqref{eq:perstep_impv}, $A_{k}$ is decreasing almost surely for all $k$. Thus $\E[A_k]$ is decreasing for all $k.$
    Then we have two different recursions
    \begin{align*}
    \begin{cases}
        \E[A_{n+1}]-\E[A_n]\leq -\E[A_n]^{3/2}/(2c_2) & \text{if } n\leq T\\
        \E[A_{n+1}]-\E[A_{n}]\leq -\E[A_{n}]^{2}/(2c_1) & \text{if } n> T\\
    \end{cases}.    
    \end{align*}
    Using similar techniques found in \cite{nesterov2006cubic,mishchenko2023regularized}, the two recursions above give
    \begin{align*}
    \begin{cases}
        \displaystyle\frac{1}{\sqrt{\E[A_{n+1}]}}-\frac{1}{\sqrt{\E[A_{n}]}}\geq \frac{1}{4c_2} & \text{if } n\leq T\\
        \displaystyle\frac{1}{\E[A_{n+1}]}-\frac{1}{\E[A_n]}\geq \frac{1}{2c_1} & \text{if } n> T\\
    \end{cases}
    \end{align*}
    The above arguments hold only for indices $n\in \mathcal{I}_{\infty}$. We can extend the recursive inequalities to all indices by working with the subsequence of indices in $\mathcal{I}_{\infty}$. Let us enumerate the elements in $\mathcal{I}_{\infty}$ as $i_{0}<i_{1}<\cdots$. From monotonic decreasing of $\E[A_k]$ and the fact that $i_{t+1}\geq i_{t}+1$, it follows that
    \begin{align*}
    \begin{cases}
        \displaystyle\frac{1}{\sqrt{\E[A_{i_{t+1}}]}}-\frac{1}{\sqrt{\E[A_{i_t}]}}\geq \frac{1}{4c_2} & \text{if } i_t\leq T\\
        \displaystyle\frac{1}{\E[A_{i_{t+1}}]}-\frac{1}{\E[A_{i_t}]}\geq \frac{1}{2c_1} & \text{if } i_t> T\\
    \end{cases}
    \end{align*}
    For any fixed $n\geq 1$, we consider three different cases.

    \textbf{Case I: $|\mathcal{I}_n|\leq n/2$.} By the definition of $\mathcal{I}_n$ and \eqref{eq:stability2_simplified}, we have
    \[
   \E[\|\nabla f(\param_n)\|]\leq \left(\frac{3}{2(1\wedge c)}\right)^{n/2}\left(\frac{(1\wedge c)}{3}\right)^{n/2}\|\nabla f(\param_0)\|=\left(\frac{1}{2}\right)^{n/2}\|\nabla f(\param_0)\|.
    \]
    By $L$-smoothness, $\| \nabla f(\param_{0}) \| \le \sqrt{2L (f(\param_{0})-f^{*}) }.$ Combine this with \eqref{eq:objective_gap_upper_bd_convex} we have a linear convergence
    \[
    \E[A_{k}]\leq D\sqrt{2LA_0}\left(\frac{1}{\sqrt{2}}\right)^n. 
    \]  
    
    \textbf{Case II: $|\mathcal{I}_n|>n/2$ and $n\leq T$}.
    
    Let $K$ be the largest integer so that $i_{K}\in \mathcal{I}_n$, then $K>n/2.$ Since $i_t\leq T$ for all $t\leq K$, we have
    \[
    \frac{1}{\sqrt{\E[A_{i_{t+1}}]}}\geq \sum_{t=0}^{K-1}\frac{1}{{\sqrt\E[A_{i_{t+1}}]}}-\frac{1}{\sqrt{\E[A_{i_{t}}}]}\geq \frac{K}{4c_2}\geq \frac{n}{8c_2}.
    \]
    Thus follows $\E[A_n]\leq \E[A_{i_K}] \le 64c_{2}^{2}n^{-2}$. 

    Combine Case I and Case II we have for any $n\leq T$, we have
    \[
    \E[A_n]\leq \max\left\{\frac{\sqrt{2LA_0}n^2}{2^{n/2}},64c_2^2\right\}\frac{1}{n^2}=:C_n\frac{1}{n^2},
    \]
    where $c_2^2=\left(\frac{81}{2(1\wedge c^3)}\right)^2D^3L_H.$

     \textbf{Case III: $|\mathcal{I}_n|>n/2$ and $n> T$}. 
    
     Still let $K$ be the largest integer so that $i_{K}\in \mathcal{I}_n$, then $K>n/2.$ Also let $1\leq \tau\leq K$ be the largest integer so that $i_{\tau}\leq T,$ and $\tau=0$ if such an integer does not exist.
     Then
    \[
    \frac{1}{\E[A_{i_K}]}-\frac{1}{\E[A_{i_{\tau}}]}=\sum_{t=\tau}^{K-1}\frac{1}{{\E[A_{i_{t+1}}]}}-\frac{1}{{\E[A_{i_{t}}}]}\geq \frac{K-\tau}{2c_1}\geq \frac{n/2-\tau}{2c_1},
    \]
    and $\E[A_{i_{\tau}}]\le C_n/\tau^{2}$ .
    Then
    \[
    \frac{1}{\E[A_{i_K}]}\geq \frac{n/2-\tau}{2c_1}+\frac{1}{\E[A_{i_{\tau}}]}\geq \frac{n/2-\tau}{2c_1}+\frac{\tau^2}{C_n}\geq n\min \{\frac{1}{4c_1},\frac{1}{2C_n}\}
    \]
    which gives $\E[A_n]\leq \E[A_K]\le \max\{4c_1,2C_n\}n^{-1}$ where $C_n$ is defined in Case II and $c_1=\frac{81}{2(1\wedge c^3)}D^2\overline{\eps}_{\E}.$ Thus we have proved the claimed two phase global convergence.
\end{proof}

\section{Local convergence}
\label{sec:local_pf}
In this section, we prove Theorem \ref{thm:local_unified} and Corollary  \ref{cor:inexact}.
Throughout this section, we assume \ref{assumption:A1} and \ref{assumption:A2} hold.
We first give a remark on equivalent formulations of Assumption \ref{assumption:A2}. 

\begin{remark}\label{rmk:PL}
Fix a local minimizer $\param_{*}$ of a $C^{2}$ objective $f$.  
Rebjock and Boumal \cite{rebjock2024fast} showed that the following conditions of PL, \textit{error bound} (EB),and \textit{quadratic growth} (QG) conditions are `essentially' equivalent: 
\begin{align}\label{eq:local_landscape_assumptions}
        \textup{$\mu$-PL near $\param_*$}&: \textup{$ f(\param)-f(\param_*)\leq \frac{1}{2\mu}\|\nabla f(\param)\|^2$;} \\
        \textup{$\mu$-EB near $\param_*$} &: \mu\ \textup{$\text{dist}(\param,\mathcal{S})\leq \|\nabla f(\param)\|$;} \nonumber \\
         \textup{ $\mu$-QG near $\param_*$} &: \textup{$f(\param)-f(\param_*)\geq \frac{\mu}{2}\text{dist}(\param,\mathcal{S})^2$}.  \nonumber
    \end{align}
    Here, `essentially' means that the constant and the neighborhood can degrade when one deduces one condition from another. {More precisely, assuming the objective $f\in C^2$, from table 1 of \cite{rebjock2024fast}, local $\mu$-PL implies local $\mu$-EB, implies local $\mu$-QG; however in the opposite direction local $\mu$-QG implies local $\mu_1$-EB, implies local $\mu_2$-PL where $\mu_2\le \mu_1\le \mu$. And one can take $\mu_2$ arbitrarily close to $\mu_1$, and $\mu_1$ to $\mu$, by shrinking the neighborhood. Thus, quantitatively, the $\mu$-QG condition we assumed in \ref{assumption:A2} is the weakest among the three.}
\end{remark}

When the objective is merely $\mu$-QG rather than $\mu$-strongly convex, we might encounter Hessians that are not positive definite. The lemma below is convenient for bounding the length of a step made by an algorithm in such a situation.

\begin{lemma}\label{lem:ABCv}
    Consider $d\times d$ matrices $\A,\B, \mathbf{C}$ all PSD, where $\A$ is strictly positive definite. Suppose $V\subset \R^{d}$ is a subspace so that $\|\B\mathbf{v}\|\geq \mu\|\mathbf{v}\|$ for all $\mathbf{v}\in V.$ Then,
    \[
    \|\A\mathbf{C}(\A+\B)^{-1}\mathbf{v}\|\leq \|\mathbf{C}\|_2\frac{\lambda_{\max}(\A)}{\lambda_{\max}(\A)+\mu}\|\mathbf{v}\|.
    \]
\end{lemma}
\begin{proof}
    Denote $\bP$ the orthogonal projection onto $V.$ Then when restricted on $V$, 
    \[
    \|\mathbf{C}\|_{2}\A(\A+\mu\bP)^{-1}\succcurlyeq\A\mathbf{C}(\A+\B)^{-1}.
    \]
    So it suffices to show $\|\A(\A+\mu\bP)^{-1}\mathbf{v}\|\leq \frac{\lambda_{\max}(\A)}{\lambda_{\max}(\A)+\mu}\|\mathbf{v}\|.$ Under an orthonormal basis of $\R^d$ whose first $\text{dim}(V)$ vectors span $V$, one can write
    \[
    \A=\begin{bmatrix}
        \A_{11} & \A_{12}\\
        \A_{21} & \A_{22},
    \end{bmatrix}\quad
    \bP=\begin{bmatrix}
        \I_{\text{dim}(V)} & \mathbf{O}\\
        \mathbf{O} & \mathbf{O},
    \end{bmatrix}\quad
    \mathbf{v}=\begin{bmatrix}
        \widetilde{\mathbf{v}} \\
        \mathbf{0}
    \end{bmatrix}.
    \]
    Apply the inversion formula in section 0.7.3 in \cite{horn2012matrix}, under the new orthonormal basis one gets
    \[
    \A(\A+\mu\bP)^{-1}\mathbf{v}=\begin{bmatrix}
        \mathbf{S}(\mathbf{S}+\mu\I_{\text{dim}(V)})^{-1}\widetilde{\mathbf{v}}\\
        \mathbf{0}
    \end{bmatrix},
    \]
    where $\mathbf{S}=\A_{11}-\A_{12}\A_{22}^{-1}\A_{21}$ is the Schur complement of $\A$. Thus,
    \[
    \|\A(\A+\mu\bP)^{-1}\mathbf{v}\|\leq \frac{\lambda_{\max}(\mathbf{S})}{\lambda_{\max}(\mathbf{S})+\mu}\|\widetilde{\mathbf{v}}\|\le \frac{\lambda_{\max}(\A)}{\lambda_{\max}(\A)+\mu}\|\mathbf{v}\|,
    \]
    where we used $\lambda_{\max}(\mathbf{S})\le \lambda_{\max}(\A)$ \cite{horn2012matrix}, $\|\widetilde{\mathbf{v}}\|=\|\mathbf{v}\|$, and the fact that $x/(1+x)$ is increasing for $x\geq 0.$
\end{proof}

Next we show the key Lemma that allows up to find an appropriate reference $\param_{n,*}\in\mathcal{S}$ for any $\param_n$ sufficiently close to $\mathcal{S}.$

\begin{lemma}\label{lem:implications of A3}
Suppose $f$ satisfies \ref{assumption:A1} and \ref{assumption:A2}. If $\|\nabla f(\param_n)\|< \eta$, where $\eta>0$ is defined in \ref{assumption:A2}, then 

    there is $\param_{n,*}\in S$ so that 
    \begin{description}
        \item[i] 
            $\param_{n,*}$ is the most efficient reference local minimum to $\param_n$ in the sense that 
            \begin{align}\label{eq:choice_param_n_*}
                \param_n-\param_{n,*}=\bP_n(\param_n-\param_{n,*}),
            \end{align}
            where the operator $\bP_n$ is defined by
            \begin{align}\label{eq:def_Pn}
                \bP_n &:= \textup{the projection operator onto the range of $\nabla^2f(\param_{n,*})$},
            \end{align}
            (See Figure \ref{fig:flat_landscape} for illustration.)
        \item[ii] 
             For any $\mathbf{v}\in \text{Range}(\nabla^2f(\param_{n,*}))$, $\innp{\mathbf{v},\nabla^2f(\param_{n,*})\mathbf{v}}\geq \mu\|\mathbf{v}\|^2$.
    \end{description}
\end{lemma}

\begin{proof}
    Under the assumption \ref{assumption:A2} and that $\|\nabla f(\param_n)\|<\eta$, by Lemma 1.4, 1.5 and Corollary 2.17 of \cite{rebjock2024fast},
    \begin{description}
        \item[(a)]
             The set $\text{proj}_{\mathcal{S}}(\param_n):=\argmin_{\param_*'\in \mathcal{S}}\|\param_{n}-\param_{*}'\|$ is not empty.
        \item[(b)]
            For any $\param_*'\in \text{proj}_{\mathcal{S}}(\param_n),$ $\param_n-\param_{n,*}\in (T_{\param_{n,*}}\mathcal{S})^{\perp}$, where $T_{\param_{n,*}}\mathcal{S}$ is the tangent space of $\mathcal{S}$ at $\param_{n,*}$.
        \item[(c)]
            $\text{Ker}(\nabla^2f(\param_{n,*}))=T_{\param_{n,*}}\mathcal{S}$, and for any $\mathbf{v}\in \text{Range}(\nabla^2f(\param_{n,*})),$ $ \innp{\mathbf{v},\nabla^2f(\param_{n,*})\mathbf{v}}\geq \mu\|\mathbf{v}\|^2.$
    \end{description}
   Thus we can fix one such $\param_*'\in \text{proj}_{\mathcal{S}}(\param_n)$ and call it $\param_{n,*}.$ Then by point (b) and (c),
   \begin{align*}
       \param_n-\param_{n,*}\in (T_{\param_{n,*}}\mathcal{S})^{\perp}=\text{Ker}(\nabla^2f(\param_{n,*}))^{\perp}=\text{Range}(\nabla^2f(\param_{n,*})),
   \end{align*}
   which translates to our first assertion, $\param_n-\param_{n,*}=\bP_n(\param_n-\param_{n,*}).$ And the second part of point (c) is exactly the second assertion.
\end{proof}


\begin{wrapfigure}[10]{r}{0.23\textwidth}
	\centering
	\vspace{-0.4cm}
	\includegraphics[width=1 \linewidth]{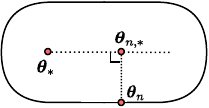}
	\vspace{-0.6cm}
	\caption{\scriptsize Example of flat local minimum $\param_{*}$. Contour represents the level curve of the objective.}
	\label{fig:flat_landscape}
\end{wrapfigure}
A key difficulty of this local convergence analysis comes from the fact that the gradient at $\param_n$ does not lie in the range of the Hessian. We overcome this by choosing a proper reference local minimum $\param_{n,*}$ given by lemma \ref{lem:implications of A3}, then $\nabla f(\param_n)$ 'almost' lie in the range of $\nabla^2f(\param_{n,*})$, where the exact meaning of 'almost' is detailed in the Lem. \ref{lem:P_dom_Q} below.

\begin{lemma} \label{lem:P_dom_Q}
Suppose $f$ satisfies \ref{assumption:A1} and \ref{assumption:A2}. Assume $\|\nabla f(\param_n)\|< \eta$, where $\eta>0$ is defined in \ref{assumption:A2}. Let $\param_{n,*}$ and $\bP_{n}$ be the ones given by Lem. \ref{lem:implications of A3} (see \eqref{eq:choice_param_n_*}), and let $\bQ_{n}$ be the projection operator onto $\text{Ker}(\nabla^2f(\param_{n,*})).$ Then 
\begin{align*} 
    \|\bP_{n}\nabla f(\param_n)\|&\geq \mu \|\param_n-\param_{n,*}\|-\frac{L_H}{2}\|\param_n-\param_{n,*}\|^2, \,\, \|\bQ_{n}\nabla f(\param_n)\|\leq \frac{L_H}{2}\|\param_n-\param_{n,*}\|^2.
\end{align*}
Furthermore, let $A$ be the number so that $\|\nabla f(\param_n)\|= \frac{(A\mu)^2}{L_H}$. If $A<\sqrt{2}$, then we have
\begin{align}
    \|\bQ_{n} \nabla f(\param_n)\| &\leq \frac{L_H}{2\mu^2}\left(1-\frac{A^2}{2}\right)^{-2}\|\bP_{n} \nabla f(\param_n) \|^2. \label{eq:P_dom_Q}
\end{align}
\end{lemma}

\begin{proof}
    By Taylor's theorem,
   \begin{align*}
       \nabla f(\param_n)&=\nabla f(\param_n) -\nabla f(\param_{n,*})=\int_{0}^{1}\nabla^2f(\param_{n,*}+t(\param_n-\param_{n,*}))dt\  (\param_n-\param_{n,*}).
   \end{align*}
   Denote $\bR(n,t):=\nabla^2f(\param_{n,*}+t(\param_n-\param_{n,*}))-\nabla^2f(\param_{n,*})$. 
   Its operator norm satisfies $\|\bR(n,t)\|_{op}\leq tL_H\|\param_n-\param_{n,*}\|$ by Lipschitz-continuity of the Hessian. Then we have
   \begin{align*}
       \bP_{n}\nabla f(\param_n)&=\bP_{n}\nabla^2f(\param_{n,*})(\param_n-\param_{n,*})+\bP_{n}\int_{0}^{1}\bR(n,t)dt \ (\param_n-\param_{n,*})\\
       &=\nabla^2 f(\param_{n,*})\bP_{n}(\param_n-\param_{n,*})+\bP_{n}\int_{0}^{1}\bR(n,t)dt \ (\param_n-\param_{n,*})\\
        &=\nabla^2 f(\param_{n,*})(\param_n-\param_{n,*})+\bP_{n}\int_{0}^{1}\bR(n,t)dt \ (\param_n-\param_{n,*})
   \end{align*}
   
   where the last line holds from the first part of Lemma \ref{lem:implications of A3}, and the previous line holds because $\nabla^2f(\param_{n,*})$ and $\bP_{n}$ commutes since they share the same frame of orthogonal eigenvectors. Now apply the norm on both sides and use the second part of Lemma \ref{lem:implications of A3}, we get the estimate
   \begin{align}\label{eq:proj_grad_norm}
        \|\bP_{n}\nabla f(\param_n)\|&\geq \mu \|\param_n-\param_{n,*}\|-\frac{L_H}{2}\|\param_n-\param_{n,*}\|^2.
   \end{align}
   Similarly we have
   \begin{gather}\label{eq:ortho_grad_norm}
   \begin{aligned}
       \|\bQ_{n}\nabla f(\param_n)\|&=\left\|\nabla^2 f(\param_{n,*})\bQ_{n}(\param_n-\param_{n,*})+\bQ_{n}\int_{0}^{1}\bR(n,t)dt \ (\param_n-\param_{n,*})\right\|\\
       &= \left\|\bQ_{n}\int_{0}^{1}\bR(n,t)dt \ (\param_n-\param_{n,*}) \right\| \leq \frac{L_H}{2}\|\param_n-\param_{n,*}\|^2
   \end{aligned}
   \end{gather}
   where the second line follows because $\text{Range}(\bQ_{n})=\text{Ker}(\nabla^2f(\param_{n,*})).$
   Recall that $\|\nabla f(\param_n)\|= \frac{(A\mu)^2}{L_H}$ with $A<\sqrt{2}.$ By \cite{rebjock2024fast}, $\mu$-PL implies $\mu$-EB, we have 
   \[
   \|\param_n-\param_{n,*}\|\leq \frac{\|\nabla f(\param_n)\|}{\mu}=\frac{A^2\mu}{L_H}\leq \frac{2\mu}{L_H}.
   \]
   Then using \eqref{eq:proj_grad_norm}, one can lower bound $ \|\bP_{n}\nabla f(\param_n)\|$ only using a linear term as
   \[
   \|\bP_{n}\nabla f(\param_n)\|\geq \left(1-\frac{A^2}{2}\right)\mu\|\param_n-\param_{n,*}\|
   \]
   Plug this into \eqref{eq:ortho_grad_norm} we have $\|\bQ_{n}\nabla f(\param_n)\|\leq \frac{L_H}{2\mu^2}\left(1-\frac{A^2}{2}\right)^{-2}\|\bP_{n}\nabla f(\param_n)\|^2$.
\end{proof}

Another difficulty of this local convergence analysis comes from the fact that if $\param_n$ is close to some $\mathcal{S}$ that has rank deficiency, the preconditioning matrix $\hat{\B}_n+\lambda_n\mathbf{I}$ can be very ill-conditioned since its smallest eigenvalue could approach zero.
In the next lemma, we show that despite that, the size of $\|(\hat{\B}_n+\lambda_n\mathbf{I})^{-1}\nabla f(\param_n)\|$ is well controlled due to the special geometrical alignment of $\nabla f(\param_n)$ shown in lemma \ref{lem:P_dom_Q}.

\begin{lemma}\label{lem:bound_beta}
Suppose $f$ satisfies \ref{assumption:A1} and \ref{assumption:A2}. If $\|\nabla f(\param_n)\|< \eta$, where $\eta>0$ is defined in \ref{assumption:A2}. Let $A$ be the number so that $\|\nabla f(\param_n)\|=\frac{(A\mu)^2}{L_H}$. If further $A<0.2$, then almost surely
\begin{align}\label{eq:bounded_beta_n}
\|(\hat{\B}_n+\lambda_n\mathbf{I})^{-1}\nabla f(\param_n)\|\leq \frac{\sqrt{2}}{\mu}\|\nabla f(\param_n)\|.
\end{align}
\end{lemma}
\begin{proof}
    Let $\param_{n,*}$ be the reference local minimum found by lemma \ref{lem:implications of A3}. As in lemma \ref{lem:P_dom_Q}, we denote
   \begin{align*}
   \bR_n:=\nabla^2f(\param_{n})-\nabla^2f(\param_{n,*}), \qquad  \D_n:=\hat{\B}_n-\nabla^2f(\param_n)=\bE_n+\lambda_n\mathbf{I}
   \end{align*}
   Notice we can rewrite $\hat{\B}_n+\lambda_n\mathbf{I}=\nabla^2f(\param_{n,*})+\D_n-\bR_n$, we can write
       \begin{align}\label{eq:lem3_2}
       &(\hat{\B}_n+\lambda_n\mathbf{I})^{-1}\nabla f(\param_n)\nonumber\\
       &=(\nabla^2f(\param_n)+\D_n)^{-1}\nabla f(\param_n)\nonumber\\
       &= \left(\nabla^2f(\param_{n,*})+\D_n-\bR_n\right)^{-1}(\bP_{n}\nabla f(\param_n)+\bQ_{n}\nabla f(\param_n))\nonumber\\
       &=\left(\mathbf{I}-\left(\nabla^2f(\param_{n,*})+\D_n\right)^{-1}\bR_n\right)^{-1}\left(\nabla^2f(\param_{n,*})+\D_n\right)^{-1}(\bP_{n}\nabla f(\param_n)+\bQ_{n}\nabla f(\param_n))\nonumber\\
       &=\left(\mathbf{I}-\left(\nabla^2f(\param_{n,*})+\D_n\right)^{-1}\bR_n\right)^{-1}\left(\nabla^2f(\param_{n,*})+\D_n\right)^{-1}\bP_{n}\nabla f(\param_n)\\
       &\quad\quad +\left(\mathbf{I}-\left(\nabla^2f(\param_{n,*})+\D_n\right)^{-1}\bR_n\right)^{-1}\left(\nabla^2f(\param_{n,*})+\D_n\right)^{-1}\bQ_{n}\nabla f(\param_n)\nonumber
   \end{align}
  Now we bound the different parts of \eqref{eq:lem3_2} individually.
  First part we have
  \begin{align*}
      &\left\|\left(\mathbf{I}-\left(\nabla^2f(\param_{n,*})+\D_n\right)^{-1}\bR_n\right)^{-1}\right\|_{2}\\
      &\leq \left(1-\frac{\|\bR_n\|_{2}}{\lambda_{\min}(\D_n)}\right)^{-1} \leq \left(1-\frac{L_H\|\param_{n}-\param_{n,*}\|}{\lambda_{\min}(\D_n)}\right)^{-1} \leq \left(1-\frac{L_H\|\param_{n}-\param_{n,*}\|}{\sqrt{L_H\|\nabla f(\param_n)\|}}\right)^{-1},
  \end{align*}
  where the last line follows from the fact that $\D_n=\bE_n+\lambda_n\mathbf{I}\succcurlyeq \sqrt{L_H\|\nabla f(\param_n)\|}\mathbf{I}$.
  Recall that $\param_{n,*}$ is $\mu$-PL, that is $\|\nabla f(\param_n)\|^2\geq 2\mu(f(\param_n)-f(\param_{n,*})).$ By \cite{otto2000generalization}, $\param_{n,*}$ is also $\mu$-QG, that is
$f(\param_n)-f(\param_{n,*})\geq \frac{\mu}{2}\|\param_n-\param_{n,*}\|^2. $
Combine the two inequalities above we have
    \begin{align*}
        \|\nabla f(\param_n)\|&\geq \mu\|\param_n-\param_{n,*}\|.
    \end{align*}
Continuing the computation we have,
    \begin{align*}
         &\hspace{-1cm}\left\|\left(\mathbf{I}-\left(\nabla^2f(\param_{n,*})+\D_n\right)^{-1}\bR_n\right)^{-1}\right\|_{2} \\
         &\leq\left(1-\frac{L_H\|\param_n-\param_{n,*}\|}{\sqrt{L_H\|\nabla f(\param_n)\|}}\right)^{-1} \leq \left(1-\frac{\sqrt{L_H\|\nabla f(\param_n)\|}}{\mu}\right)^{-1}\\
          &\leq \left(1-\frac{\sqrt{L_H\|\nabla f(\param_n)\|}}{\mu}\right)^{-1} =\frac{1}{1-A},
    \end{align*}
    where the last line follows $\|\nabla f(\param_n)\|=(A\mu)^2/L_H.$
To summarize, the first part of \eqref{eq:lem3_2} has upper bound
\begin{align}\label{eq:upper_beta_1}
    \left\|\left(\mathbf{I}-\left(\nabla^2f(\param_{n,*})+\D_n\right)^{-1}\bR_n\right)^{-1}\right\|_{2}&\leq \frac{1}{1-A}.
\end{align}
Using lemma \ref{lem:implications of A3}, the second part of \eqref{eq:lem3_2} can be easily bounded by
\begin{align}\label{eq:upper_beta_2}
   \|(\nabla^2f(\param_{n,*})+\D_n)^{-1}\bP_n\nabla f(\param_n)\|\leq \frac{1}{\mu+\lambda_{\min}(\D_n)}\|\bP_n\nabla f(\param_n)\|\leq \frac{1}{\mu}\|\bP_n\nabla f(\param_n)\|.
\end{align}
The third part is the key part that uses the lemma \ref{lem:P_dom_Q}, according to which,
\begin{align*}
\left\|\left(\nabla^2f(\param_{n,*})+\D_n\right)^{-1}\bQ_{n}\nabla f(\param_n)\right\|_{2}\leq \frac{L_H}{2\mu^2\lambda_{\min}(\D_n)}\left(1-\frac{A^2}{2}\right)^{-2}\|\bP_{n}\nabla f(\param_n)\|^2.
\end{align*}
Using $\lambda_{\min}(\D_n)\geq \sqrt{L_H\|\nabla f(\param_n)\|}$ and $\|\nabla f(\param_n)\|\leq (A\mu)^2/L_H$, we get
\begin{align}\label{eq:upper_beta_3}
\left\|\left(\nabla^2f(\param_{n,*})+\D_n\right)^{-1}\bQ_{n}\nabla f(\param_n)\right\|_{2}\leq \frac{A}{2\mu}\left(1-\frac{A^2}{2}\right)^{-2}\|\bP_{n}\nabla f(\param_n)\|
\end{align}
Then combine \eqref{eq:upper_beta_1}, \eqref{eq:upper_beta_2}, and \eqref{eq:upper_beta_3} into \eqref{eq:lem3_2}, we have
\begin{align*}
    \|(\hat{\B}_n+\lambda_n\mathbf{I})^{-1}\nabla f(\param_n)\|&\leq \frac{1}{1-A}\left(1+\frac{A}{2}\left(1-\frac{A^2}{2}\right)^{-2}\right)\frac{1}{\mu}\|\nabla f(\param_n)\|\\
    &\leq \frac{\sqrt{2}}{\mu}\|\nabla f(\param_n)\|,
\end{align*}
where the last line follows from the choice of $A< 0.2.$
\end{proof}

We are now ready to prove our local convergence result stated in Theorem \ref{thm:local_unified}. 

\begin{proof}[\textbf{Proof of Theorem} \ref{thm:local_unified}]
    Recall the next iterate is given by $ \param_{n+1}=\param_n-(\hat{\B}_n+\lambda_n\mathbf{I})^{-1}\nabla f(\param_n),$ and similar to the previous Lemma we denote
    \[
    \D_n:=\hat{\B}_n-\nabla^2f(\param_n)=\bE_n+\lambda_n\mathbf{I}.
    \]
    For notational simplicity, we also denote 
    \[
    \boldsymbol{\delta}_n:=(\hat{\B}_n+\lambda_n\mathbf{I})^{-1}\nabla f(\param_n)
    \]
    Then by standard argument using line integral, 
    we can express the gradient at $\param_{n+1}$ as 
    \begin{align}\label{eq:next_grad}
        \nabla f(\param_{n+1})&=\nabla f(\param_n-\boldsymbol{\delta}_n)\nonumber\\
        &=\nabla f(\param_n)-\int_{0}^{1}\nabla^2f(\param_n-t\boldsymbol{\delta}_n)dt\ \boldsymbol{\delta}_n\nonumber\\
        &=\B_n\boldsymbol{\delta}_n-\int_{0}^{1}\nabla^2f(\param_n-t\boldsymbol{\delta}_n)dt\ \boldsymbol{\delta}_n\nonumber\\
        &= \nabla^2f(\param_n)\boldsymbol{\delta}_n+\D_n\boldsymbol{\delta}_n-\int_{0}^{1}\nabla^2f(\param_n-t\boldsymbol{\delta}_n)dt\ \boldsymbol{\delta}_n\nonumber\\
        &=\D_n\boldsymbol{\delta}_n
        -\int_{0}^{1}\left[\nabla^2f(\param_n)-\nabla^2f(\param_n-t\boldsymbol{\delta}_n)\right]dt\ \boldsymbol{\delta}_n.
    \end{align}
    For the quadratic integral term in \eqref{eq:next_grad}, using $L_H$-Lipschtz continuity of the Hessian, we can bound it's norm 
    \begin{align*}
    \left\|\int_{0}^{1}\left[\nabla^2f(\param_n)-\nabla^2f(\param_n-t\boldsymbol{\delta}_n)\right]dt\ \boldsymbol{\delta}_n\right\| \leq \frac{L_H}{2}\|\boldsymbol{\delta}_n\|^2 \overset{\eqref{eq:bounded_beta_n}}{\leq} \frac{L_H}{\mu^2}\|\nabla f(\param_n)\|^2.
    \end{align*}
    For the linear term in \eqref{eq:next_grad}, we decompose the gradient again
    \begin{align}\label{eq:DB}
\D_n\boldsymbol{\delta}_n&=\D_n(\hat{\B}_n+\lambda_n\mathbf{I})^{-1}\nabla f(\param_{n})\nonumber\\
&=\D_n(\hat{\B}_n+\lambda_n\mathbf{I})^{-1}\bP_n\nabla f(\param_{n})+
\D_n(\hat{\B}_n+\lambda_n\mathbf{I})^{-1}\bQ_n\nabla f(\param_n)
    \end{align}
    Denote $A=\sqrt{L_H\|\nabla f(\param_n)\|/\mu}$, then by \eqref{eq:P_dom_Q}
    \begin{align}\label{eq:DBQ}
    \|\D_n(\hat{\B}_n+\lambda_n\mathbf{I})^{-1}\bQ_n\nabla f(\param_{n})\|\leq \frac{L_H}{2\mu^2}\left(1-\frac{A^2}{2}\right)^{-2}\|\nabla f(\param_n)\|^2\leq \frac{L_H}{\mu^2}\|\nabla f(\param_n)\|^2.,
    \end{align}
    where the last line follows from the assumption that $A\leq \gamma_1<0.2$.
    Recall \eqref{eq:lem3_2}, we have
    \begin{align*}
    &\hspace{0.2in}\D_n(\hat{\B}_n+\lambda_n\mathbf{I})^{-1}\bP_n\nabla f(\param_{n})\\
    &=\D_n\left(\mathbf{I}-\left(\nabla^2f(\param_{n,*})+\D_n\right)^{-1}\bR_n\right)^{-1}\left(\nabla^2f(\param_{n,*})+\D_n\right)^{-1}\bP_{n}\nabla f(\param_n)
    \end{align*}
    By Lemma \ref{lem:ABCv} and that $\|\D_n\|_2\leq \|\bE_n\|_2+\lambda_n$, together with the operator norm bound \eqref{eq:upper_beta_1}, we have
    \begin{align}\label{eq:DBP}
    \|\D_n(\hat{\B}_n+\lambda_n\mathbf{I})^{-1}\bP_n\nabla f(\param_{n})\|\leq \frac{1}{1-A}\frac{\|\bE_n\|_2+\lambda_n}{\mu+\|\bE_n\|_2+\lambda_n}\|\nabla f(\param_n)\|.
    \end{align}
    Combine \eqref{eq:DBP} and \eqref{eq:DBQ} into \eqref{eq:DB} and add the quadratic integral term, and recall the definition $\overline{\lambda_n}:=\overline{\eps}+\lambda_n\geq \|\bE_n\|_2+\lambda_n$ we have
    \begin{align}\label{eq:grad_norm_contraction}
    \|\nabla f(\param_{n+1})\|&\leq \frac{1}{1-A}\frac{\overline{\lambda_n}}{\mu+\overline{\lambda_n}}\|\nabla f(\param_n)\|+ \frac{2L_H}{\mu^2}\|\nabla f(\param_n)\|^2\\
    &\leq \frac{\overline{\lambda_n}}{\mu+\overline{\lambda_n}}\|\nabla f(\param_n)\|+\frac{5}{4}A\frac{\overline{\lambda_n}}{\mu+\overline{\lambda_n}}\|\nabla f(\param_n)\|+ \frac{2L_H}{\mu^2}\|\nabla f(\param_n)\|^2\\
    &=\frac{\overline{\lambda_n}}{\mu+\overline{\lambda_n}}\|\nabla f(\param_n)\|+\frac{5}{4}\frac{\sqrt{L_H}}{\mu}\frac{\overline{\lambda_n}}{\mu+\overline{\lambda_n}}\|\nabla f(\param_n)\|^{3/2}+ \frac{2L_H}{\mu^2}\|\nabla f(\param_n)\|^2
    \end{align}
    In order for this to be a proper contraction for $A<\min\left\{0.2,\frac{20}{33}\frac{\mu}{1.2\mu+\overline{\eps}}\right\}$, rewrite
    \begin{align*}
     &\|\nabla f(\param_{n+1})\|\\
     &\leq \frac{\overline{\lambda_n}}{\mu+\overline{\lambda_n}}\|\nabla f(\param_n)\|+\frac{5}{4}\frac{\sqrt{L_H}}{\mu}\frac{\overline{\lambda_n}}{\mu+\overline{\lambda_n}}\|\nabla f(\param_n)\|^{3/2}+ \frac{2L_H}{\mu^2}\|\nabla f(\param_n)\|^2\\
     &\leq \left(\frac{\eps_n+0.2\mu}{\eps_n+1.2\mu}+\frac{5}{4}\frac{\sqrt{L_H}}{\mu}\frac{\eps_n+0.2\mu}{\eps_n+1.2\mu}\|\nabla f(\param_n)\|^{1/2}+\frac{2L_H}{\mu^2}\|\nabla f(\param_n)\|\right)\|\nabla f(\param_n)\|\\
     &\leq \left(a+\frac{5}{4}aA+2A^2\right)\|\nabla f(\param_n)\|,
    \end{align*}
    where in the last line we defined $a=\frac{\overline{\eps}+0.2\mu}{1.2\mu+\overline{\eps}}$ and used $\|\nabla f(\param_n)\|=(A\mu)^2/L_H$, then
    \begin{align*}
        a+\frac{5}{4}aA+2A^2
        &\leq (1+5A/4)a+2A/5 \quad \text{(Since $A<0.2$)}\\
        &=(5a/4+2/5)A+a\\
        &\leq \frac{33}{20}A+a \leq \frac{\mu}{1.2\mu+\overline{\eps}}+\frac{\overline{\eps}+0.2\mu}{1.2\mu+\overline{\eps}}\leq 1,
    \end{align*}
    where the last line used $A<\frac{20}{33}\frac{\mu}{1.2\mu+\overline{\eps}}$.
\end{proof}

Next we use the Theorem \ref{thm:local_unified} to prove the Corollary \ref{cor:inexact}.
\begin{proof}[\textbf{Proof of Corollary} \ref{cor:inexact}]
Notice that $N_2\geq N_1$ almost surely. So, by Thm. \ref{thm:local_unified}, on the event $n\geq N_2$, almost surely,
\begin{align*}
\|\nabla f(\param_{n+1})\|&\leq \frac{\overline{\lambda_n}}{\mu+\overline{\lambda_n}}\|\nabla f(\param_n)\|+\frac{5}{4}\frac{\sqrt{L_H}}{\mu}\frac{\overline{\lambda_n}}{\mu+\overline{\lambda_n}}\|\nabla f(\param_n)\|^{3/2}+ \frac{2L_H}{\mu^2}\|\nabla f(\param_n)\|^2.
\end{align*}
Since $\gamma_2<\gamma_1<0.2,$ $\sqrt{L_H\|\nabla f(\param_n)\|}<0.2\mu.$ So, $\overline{\lambda_n}=\eps_n+\sqrt{L_H\|\nabla f(\param_n)\|}\leq \eps_n+0.2\mu.$
Denote 
\[
 a=\frac{\eps_n+0.2\mu}{\eps_n+1.2\mu}\geq \frac{\overline{\lambda_n}}{\mu+\overline{\lambda_n}}, \quad  A=\sqrt{L_H\|\nabla f(\param_n)\|/\mu}<\gamma_2,
 \]
 similar to the end of the proof of Thm. \ref{thm:local_unified} we have almost surely,
\begin{align*}
\|\nabla f(\param_{n+1})\|&\leq \left(a+\frac{5}{4}aA+2A^2\right)\|\nabla f(\param_n)\|\\
&\leq (\frac{33}{20}A+a)\|\nabla f(\param_n)\| < (\frac{33}{20}\gamma_2+a)\|\nabla f(\param_n)\|.
\end{align*}
Since $\gamma_2=\delta\gamma_1<\delta\frac{20}{33}\frac{\mu}{1.2\mu+\overline{\eps}},$ continue the calculation we have
\begin{align*}
\|\nabla f(\param_{n+1})\| &< \frac{33}{20}\gamma_2+a 
\leq (\frac{\delta\mu}{1.2\mu+\overline{\eps}}+\frac{\eps_n+0.2\mu}{\eps_n+1.2\mu})\|\nabla f(\param_n)\|\\
&\leq (\frac{\delta\mu}{1.2\mu+{\eps_n}}+\frac{\eps_n+0.2\mu}{\eps_n+1.2\mu})\|\nabla f(\param_n)\| =\frac{(\delta+0.2)\mu+\eps_n}{1.2\mu+\eps_n}\|\nabla f(\param_n)\|.
\end{align*}
Then, on the event $n> N_2$, taking the conditional expectation and using Jensen's inequality gives 
\begin{align*}
\E[\|\nabla f(\param_{n+1})\|\ |\  n>N_2]&\leq \frac{(\delta+0.2)\mu+\overline{\eps}_{\E}}{1.2\mu+\overline{\eps}_{\E}}\E[\|\nabla f(\param_{n})\|\ |\  n>N_2]\\
&=\left(1-(1-\delta)\frac{\mu}{1.2\mu+\overline{\eps}_{\E}}\right)\E[\|\nabla f(\param_{n})\|\ |\  n>N_2].
\end{align*}
Now iterating the inequality above $m\geq 1$ times gives 
\[
\E[\|\nabla f(\param_{n+m})\|\ |\  n>N_2]=\left(1-(1-\delta)\frac{\mu}{1.2\mu+\overline{\eps}_{\E}}\right)^m\E[\|\nabla f(\param_{n})\|\ |\  n>N_2].
\]
\end{proof}

\section{Detailed implementation of RON with RPC}
\label{sec:alg_detail}
Here we give the detailed algorithmic implementation of our RON \eqref{eq:RON_high_level} using rank-$k$ RPC.
\begin{algorithm}[H]
	\small
	\caption{Regularized Overestimated Newton with RPCholesky}
	\label{algorithm:RON}
	\begin{algorithmic}[1]
		 \State \textbf{Input:} $\param_{0}\in\Param$ (initial estimate)
        \State \textbf{Parameters:}  $M$ (number of iterations); $k$ (lower rank parameter); $L_H$ (Lipschitz constant of the Hessian)
		\State \quad \textbf{for} $n=0,1,\cdots,M$ \textbf{do} 
		\State \quad\quad $\mathbf{F}_n,\eps_n\xleftarrow[]{}$ output of Algorithm \ref{algorithm:RPC_simple} with input ($\param_{n},k$) \hspace{1cm} ($\triangleright$ \textit{Random Pivoted Cholesky \cite{chen2022randomly}.})
        \State \quad\quad $\lambda_n\xleftarrow[]{}\eps_n+\sqrt{L_H\|\nabla f(\param_n)\|}$
        \State \quad\quad $p_{n} \leftarrow - \left(\mathbf{F}_n\mathbf{F}_n^T+\lambda_n\mathbf{I}\right)^{-1}  \nabla f(\param_{n})$ \hspace{3cm} ($\triangleright$ \textit{Use Woodburry indentity \cite{woodbury1950inverting}})
        \State \quad\quad $\param_{n+1}\xleftarrow[]{}\param_{n} + p_n$ \hspace{5cm} ($\triangleright$ \textit{Parameter update})
		\State \textbf{output} $\param_{M}$ 
	\end{algorithmic}
\end{algorithm}

In line 6 of Alg. \ref{algorithm:RON}, by applying Woodbury identity \cite{woodbury1950inverting}, we can use the following trick to calculate the search direction
	\begin{align*}
		(\mathbf{F_n}\mathbf{F_n}^T+\lambda_n \I_{N})^{-1}\nabla f(\param_n) =\lambda_n^{-1}\nabla f(\param_n)-\lambda_n^{-1}\mathbf{F}(\lambda_n\I_k+\mathbf{F}^T\mathbf{F})^{-1}\mathbf{F}^T\nabla f(\param_n).
	\end{align*}
	The dominating computations above are $N\times k$ matrices multiplying vectors of dimension $k$, $k\times N$ matrices multiplying vectors of dimension $N$, and inverting a $k\times k$ matrix. Therefore, the computational cost of each iteration is indeed $O(k^2N)$.

\begin{algorithm}[H]
	\small
	\caption{Random Pivoted Cholesky Factorization (RPC) \cite{chen2022randomly}}
	\label{algorithm:RPC_simple}
	\begin{algorithmic}[1]
		\State \textbf{Input:} $\A\in \R^{N\times N}$ ; $k\leq N$ (number of columns to be sampled); 
		\State \quad Set $\mathbf{F} \leftarrow O\in\R^{N\times k}$,  $\textbf{d}\leftarrow \textup{diagonal of $\A$}$ ; 
		\State \quad \textbf{for} $n=1,\cdots,k$ \textbf{do}
		\State \quad\quad pick $s\in\{1,...,N\}$ according to distribution $\mathbf{d}/\lVert \mathbf{d} \rVert_{1}$
		\State \quad\quad $\mathbf{col}\xleftarrow{} 
		\A[:,s]$
		\State \quad\quad $\mathbf{g}\xleftarrow{}\mathbf{col}-\mathbf{F}[:,1:n-1]\mathbf{F}[s,1:n-1]^T$
		\State \quad\quad \textbf{if} $\|\mathbf{g}\|_{1}=0$
		\State \quad\quad\quad \textbf{output} $\mathbf{F}, 0$ \hspace{1cm} (\textit{$\triangleright$ Exact factorization})
		\State \quad\quad \textbf{else}
		\State \quad\quad\quad $\mathbf{F}[:,n]\xleftarrow{}\mathbf{g}/\sqrt{\mathbf{g}[s]}$
		\State \quad\quad\quad $\mathbf{d}\xleftarrow{}\mathbf{d}-\mathbf{F}[:,n]*\mathbf{F}[:,n]$  \hspace{0.7cm} ($\triangleright$ \textit{Update the diagonal of the error matrix $\A-\mathbf{F}\mathbf{F}^T$})
		\State \quad $R \leftarrow \lVert \mathbf{d} \rVert_{1}$ \hspace{0.7cm} ($\triangleright$ \textit{Trace of the error matrix, an upper bound for $\|\A-\mathbf{F}\mathbf{F}^T\|_2$})
		\State \quad \textbf{output} $\mathbf{F}$, $R$
	\end{algorithmic}
\end{algorithm}
	We list the matrix $\mathbf{A}\in \R^{N\times N}$ for the ease of notation. In fact, this randomized factorization only requires the reading of $(k+1)N-k$ entries of $\mathbf{A}$ coming from $k$ columns of $\mathbf{A}$ and its diagonal. When the input $\mathbf{A}$ is PSD, the output $\mathbf{F}$ satisfies $\A-\mathbf{F}\mathbf{F}^T$ is PSD (see Thm 5.3 in \cite{zhang2006schur}).

If the function $f$ exhibits lower rank Hessian near a local minimum, and the number of pivoting columns, $k$, used in the RPC process exceeds the actual rank, then the lower rank Hessian approximation is exact. This is helpful for establishing the superlinear convergence in Theorem \ref{thm:local_unified}.  

\begin{theorem}[Error bound for RPC; Thm. 5.1 and Lem. 5.3 in  \cite{chen2022randomly}]\label{thm:RPC}
    Let $\A$ be a positive semi-definite  matrix. Fix \( r \in \mathbb{N} \) and \( \varepsilon > 0 \). The column
Nystr\"{o}m approximation \( \hat{\A}^{(k)} \) produced by \( k \) steps of RPC (Alg. \ref{algorithm:RPC_simple}) satisfies $\A-\hat{\A}^{(k)}\succeq \mathbf{O}$ a.s. and attains the bound $\mathbb{E}\left[ \operatorname{tr}(\A - \hat{\A}^{(k)}) \right] \leq (1 + \varepsilon) \cdot \operatorname{tr}(\A - \llbracket\A \rrbracket_r)$, 
provided that the number \( k \) of sampled columns satisfies
\begin{align}\label{eq:RPC_sample_complexity}
k \geq \frac{r}{\varepsilon} + \min \left\{ r \log \left(\frac{1}{\eps \eta}\right), r + r \log_{+} \left(\frac{2^r}{\varepsilon}\right) \right\}.
\end{align}
The relative error \( \eta \) is defined by \( \eta := \frac{\operatorname{tr}(\A - \llbracket\A \rrbracket_r)}{\operatorname{tr}(\A)} \). \( \log_{+}(x) := \max\{\log x, 0\} \) for \( x > 0 \), and the logarithm has base \( e \).

\end{theorem}

The following is a useful corollary of Theorem \ref{thm:RPC}, which states that RPC with rank parameter $k$ larger than the true rank of the matrix that it is approximating will almost surely return the perfect approximation.  

\begin{corollary}[Exact factorization by RPC]\label{cor:RPC_exact}
     Let $\A$ be a positive semi-definite  matrix with rank $r$ and let $\hat{\A}^{(k)}$ denote the the column
Nystr\"{o}m approximation \( \hat{\A}^{(k)} \) produced by \( k \) steps of RPC (Alg. \ref{algorithm:RPC_simple}). If $k\ge r$, then $\A=\hat{\A}^{(k)}$ almost surely. 
\end{corollary}

\begin{proof}
    Since $k\ge r$, we have $\eta = \frac{\operatorname{tr}(\A - \llbracket\A \rrbracket_r)}{\operatorname{tr}(\A)} =0$. Hence by choosing $\eps$ sufficiently small, the sample complexity bound \eqref{eq:RPC_sample_complexity} becomes $k\ge r$, which is already satisfied by the hypothesis. Then by Theorem \ref{thm:RPC}, we have that $\A-\hat{\A}^{(k)}$ is PSD and $\mathbb{E}\left[ \operatorname{tr}(\A - \hat{\A}^{(k)}) \right] = 0$. Thus $\A-\hat{\A}^{(k)}$ is a random PSD matrix with zero trace norm, which must be the zero matrix almost surely. 
\end{proof}

\section{Experimental results}
\label{sec:experiments}

\subsection{Entropic Optimal Transport}

We validate the efficacy of RON for the problem of computing entropic optimal transport \cite{cuturi2013sinkhorn, pavon2021data}, which has drawn significant attention from the ML community recently. Roughly speaking, given row and column marginals $(\r,\c)$, we seek to find a matrix giving the joint distribution with the least overall transportation cost. It is standard to solve its Kantorovich dual by well-known procedures such as the Sinkhorn algorithm \cite{cuturi2013sinkhorn}. 
This reads as the following smooth convex optimization problem 
\begin{align}\label{eq:typical_Lagrangian}
	\sup_{\balpha,\bbeta} \bigg( 
	\langle \r,\balpha \rangle  + \langle \c, \bbeta \rangle 
	- \langle \W , \exp(\balpha\oplus \bbeta) \rangle
	\bigg),
\end{align}
where $\W_{ij}=e^{-c(i,j)/\eps} \r(i)\c(j)$ with $c(\cdot,\cdot)$ being the cost function, $\eps>0$ is the entropic regularization parameter, and 
\(\oplus\) denotes broadcasted summation, i.e., \((\balpha \oplus \bbeta)_{ij} = \balpha_i + \bbeta_j\). One can regard the dual potential functions $\balpha$ and $\bbeta$ as the Lagrange multipliers for the row and column margin constraints, respectively. The optimal potentials $(\balpha,\bbeta)$ yield the optimal transport map via $\bZ^{\r,\c}=\exp(\balpha\oplus \bbeta)$. 

The canonical algorithm for solving \eqref{eq:typical_Lagrangian} is the \textit{Sinkhorn algorithm} \cite{cuturi2013sinkhorn}, which alternates scaling rows and columns of an exponentiated kernel matrix to match the prescribed marginals. Sinkhorn enjoys linear convergence \cite{carlier2022linear} and is computationally efficient. However, its convergence rate can deteriorate significantly when the entropic penalty is small and the dual potentials become sharp and ill-conditioned. To overcome this limitation, recent works have explored Newton-type algorithms to solve the dual of entropic optimal transport and have demonstrated improved empirical performance over Sinkhorn \cite{brauer2017sinkhorn, tang2024safe}. 

Notice that the dual objective function in \eqref{eq:typical_Lagrangian} is invariant under adding a constant to all coordinates of $\balpha$ and subtracting the same constant to all coordinates of $\bbeta$, so the minimizers of this problem form a 1-dimensional submanifold. Hence, it is not strictly convex. However, it has been recently established by Rigollet and Stromme  \cite[Prop. 12]{rigollet2025sample} that it satisfies $\mu$-PL for some $\mu$ in any bounded set of the dual variables. 
Thus our local convergence resutls (Thm. \ref{thm:local_unified}, Cor \ref{cor:exact}, Cor \ref{cor:inexact}) applies. Moreover, the Hessian of the dual objective in \eqref{eq:typical_Lagrangian} has rank at most $a+b$, where $a$ and $b$ denote the $\ell_{0}$-norm of the row and column marginals. Thus, finding optimal transport between sparse marginals is an instance of a smooth convex problem with uniformly low-rank Hessian. 


\begin{figure}[h!]
    \centering
    \includegraphics[width=\linewidth]{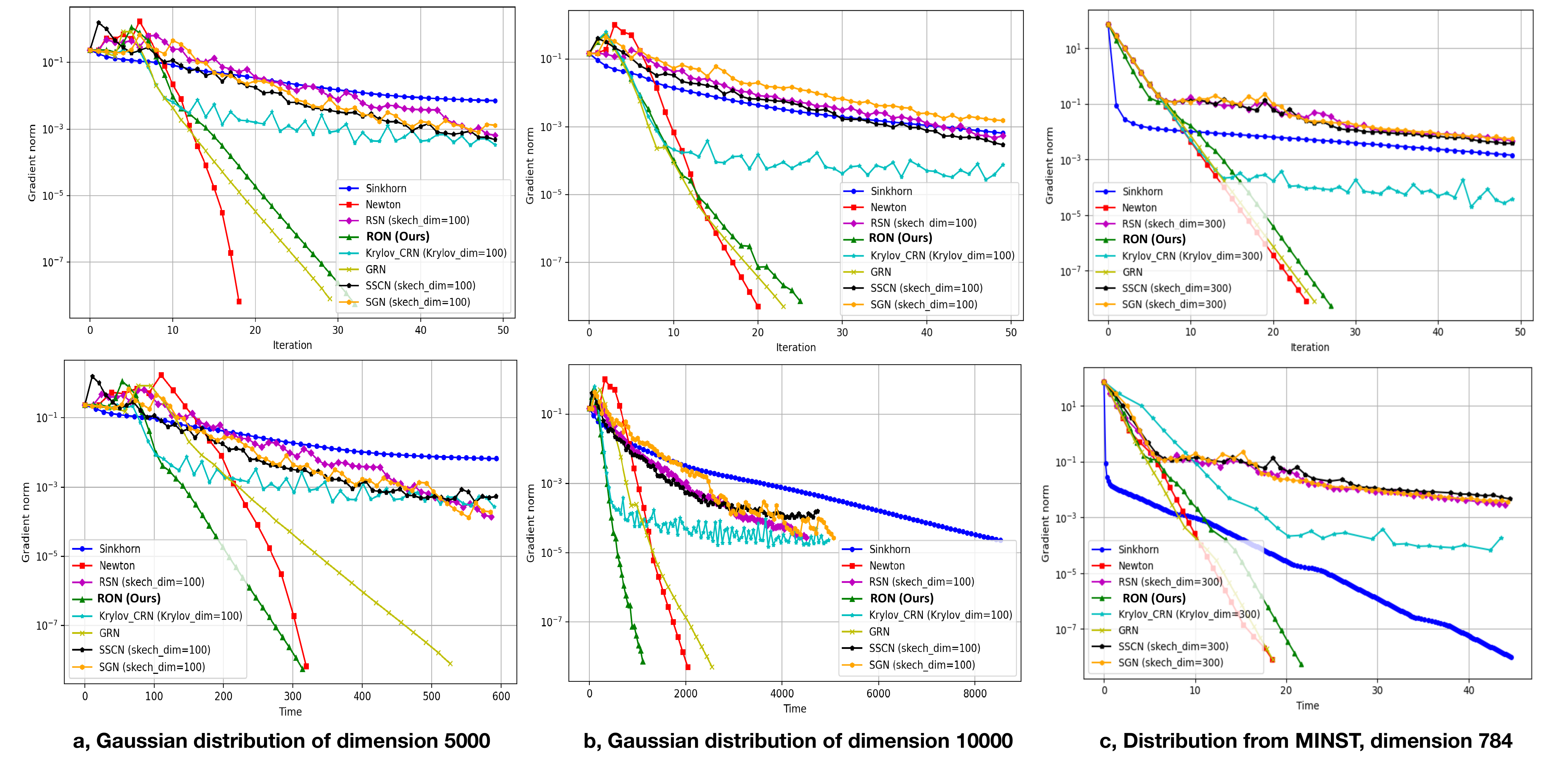}
    \vspace{-0.3cm}
    \caption{Plots of gradient norm vs iteration (top row)/time (bottom row) for solving the EOT problem by various algorithms. In column \textbf{a}, the source and target distributions used are synthesized Gaussian distributions with means $\mu=0.3,0.7$ respectively and standard deviation $\sigma=0.001$ with dimension $5000$, the cost matrix has i.i.d. uniform random variables in $[0,1].$ In column \textbf{b}, the source and target distributions and the cost matrix are similarly generated, except the dimension is $10000$. In column \textbf{c}, the source and target distributions used in the third experiment come from two normalized images in the MNIST \cite{lecun1998gradient} data set, and the cost matrix is the pixelwise $\ell_1$ distance.}
    \label{fig:experiments}
\end{figure}
We compare our proposed RON with the Hessian estimation coming from RPC with Sinkhorn, Newton, and various Newton-type algorithms. We plot the gradient norm of the dual objective function against iteration (on the top row) and time (on the bottom row) used by different algorithms. For the algorithms that reduce the Newton step to a smaller dimension $k<d$, we use the same $k$ for all such algorithms for each experiment setting. And all experiments are done on a MacBook Pro with 16GB RAM M2 chip. In columns \textbf{a},\textbf{b}, the source and target distributions are generated by first discretizing the probability density functions of two Gaussian distributions restricted on the interval $[0,1]$ and then normalizing the discretized vectors. In both experiments, the means of the two Gaussians are $0.3$ and $0.7$ with the same standard deviation $0.001$, and the same hyperparameter $k=100$ is used. The dimensions of the distributions in \textbf{a} and \textbf{b} are $5000$ and $10000$ respectively (meaning the total number of parameters to learn is $10000$ and $20000$ respectively). And the entries of the cost matrix for this set of experiments are drawn from i.i.d. Unif$[0,1].$ And in columns \textbf{c}, the source and target distributions come from two different vectorized digital images in the MNIST \cite{lecun1998gradient} data set, and the dimension of each distribution is $784$ (meaning total number of parameters to learn is $1568$), and the hypermapameter is $k=300.$ The cost matrix in this experiment comes from the pixel-wise $\ell_1$ distance. 

\subsection{Large linear inverse problems}\label{sec:experiments_app}

Here we test our algorithm on solving linear system $\A \mathbf{x}=\mathbf{b}$, where $\A\in \R^{p\times d}$ and $\mathbf{b}\in \R^{p}$ with $p\gg d$. This can be reformulated as the optimization problem below
\[
\min_{\x \in \R^{d}} \, \left[ f(\x):=\frac{1}{2}\lVert \A \x - \mathbf{b} \rVert^{2}\right].
\]
\begin{figure}[h!]
    \centering
\includegraphics[width=1\linewidth]{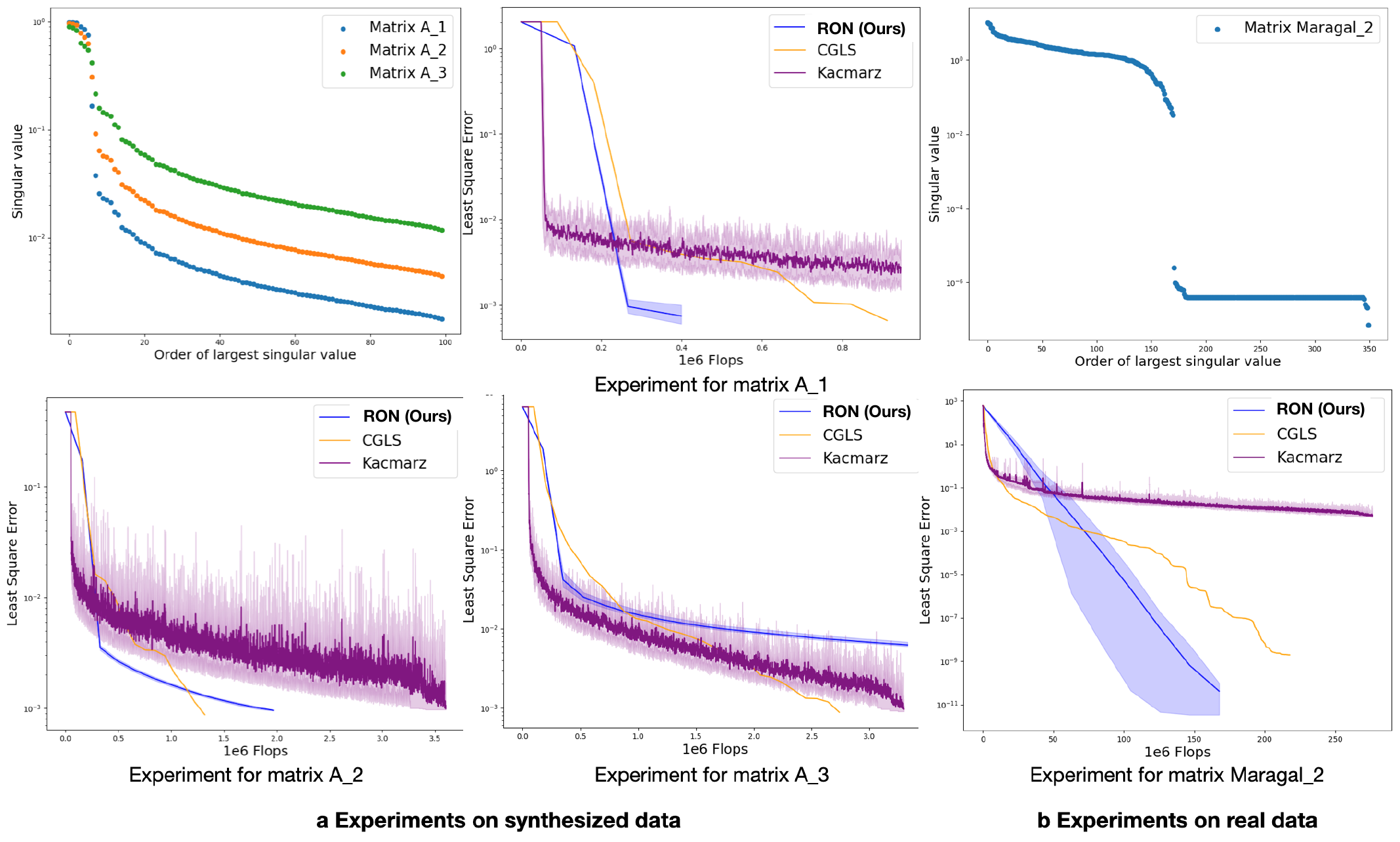} 
    \caption{Plots of singular values distributions of four matrices and the result of solving $\A\mathbf{x}=\mathbf{b}$ are plotted as least square error vs. flops used by the algorithms. Part \textbf{a}, the left-hand side, is the experiments on synthesized data. Part \textbf{b}, the right hand side, is the experiment on matrix \texttt{Maragal2} in the SuiteSparse Matrix Collection \cite{kolodziej2019suitesparse}.}
    \label{fig:experiments_linear}
\end{figure}

While this problem has a constant Hessian, it has a wide range of applications, especially in scientific computing. Also, since we can control the Hessian of the objective through the singular value profile of $\A$, it provides a nice class of problems to test our improved local linear convergence rate (Thm. \ref{thm:local_unified}, Cor.\ref{cor:inexact}). We present experimental results with both synthetic and real data. We compare our proposed RON with the Hessian estimation coming from RPC algorithm with parameter $k=20$ with two widely used linear problem solvers, randomized Kaczmarz  \cite{strohmer2009randomized} and conjugate gradient least squares (CGLS) \cite{hestenes1952methods}. We plot the least square error as a function of flops. Since the randomized Kacmarz and our RON are stochastic, we repeated the experiments 10 times and plotted the mean and shaded the region between the running maximum and minimum.

In Figure \ref{fig:experiments_linear} \textbf{a}, we present three synthetic experiments with matrices $\A_{1},\A_{2},\A_{3}\in\R^{500\times 100}$. We first sample $\A\in \R^{500\times 100}, \mathbf{b}\in \R^{500}$ from i.i.d. standard normal. Then we perform a singular value decomposition on the matrix $\A$ and get $\A=\mathbf{U}\mathbf{\Sigma}\mathbf{V}^T.$ We perturb the singular values in $\mathbf{\Sigma}$ to create  $\A_{1},\A_{2},\A_{3}$ with singular value profiles depicted in Figure \ref{fig:experiments_linear} top left. Going from $\A_{1}$ to $\A_{3}$, the singular values decay slower. Hence, with RPC being our random Hessian estimator, the expected overestimate error $\overline{\eps}_{\E}$ becomes larger by Thm. \ref{thm:RPC}. Then, according to Cor. \ref{cor:inexact}, we should expect that going from matrix $\A_1$ to $\A_3$, our algorithm would gradually lose its advantage in leveraging the low rank structure, which is validated in the following pictures of part \textbf{a}.

In Figure \ref{fig:experiments_linear} part \textbf{b} we also test our algorithm on matrix \texttt{Maragal2} 
in the SuiteSparse Matrix Collection 
\cite{kolodziej2019suitesparse}. This matrix has $555$ rows, $350$ columns, and has significant rank deficiency as seen in its singular value distribution in Figure \ref{fig:experiments_linear} top right. We used $k=171$ for RPC in this experiment. Although $k=171$ is a relatively large rank, capturing that part of the curvature structure helps to reduce the 'effective condition number' significantly, which helps our algorithm to converge very fast. It only takes 5-10 iterations for our algorithm to reach $10^{-10}$ error. We observe a superior convergence rate with our RON algorithm compared to the other benchmark methods. This highlights that our RON algorithm leverages low-rank structure in the landscape of the problem and yields faster convergence than other methods that do not leverage such geometric information. 



\section*{Acknowledgments}
Hanbaek Lyu is partially supported by NSF grant DMS-2206296.

\bibliographystyle{plain}
\bibliography{mybib}
\end{document}